\documentclass[11pt,oneside,reqno]{amsart}

\usepackage[utf8]{inputenc}
\usepackage[T1]{fontenc}
\usepackage[russian,ngerman,english]{babel}
\usepackage{lmodern}
\usepackage{hyperref}
\usepackage{amsmath}
\usepackage{amsthm}
\usepackage{amsfonts}
\usepackage{amssymb}
\usepackage{amscd}
\usepackage{amsbsy}
\usepackage{subcaption}
\usepackage{pdfpages}
\usepackage{fancyhdr}
\usepackage[hmargin=25mm,vmargin=25mm,bmargin=30mm,tmargin=25mm]{geometry}
\usepackage{geometry}
\usepackage{setspace}
\usepackage{xcolor}
\usepackage{tikz}
\usepackage{pifont}
\usepackage{graphicx}
\usepackage{tabularx}
\usepackage{epsfig}
\usepackage[all]{xy}
\usepackage{tikz}
\usepackage[normalem]{ulem}

\usetikzlibrary{matrix}

  \usepackage{fixmath}

  \usepackage{appendix}

  \usepackage{enumerate}

  \usepackage[sort, numbers]{natbib}
  \usepackage{bibentry}
  \newtheorem{theorem}{Theorem}[section]

    \newtheorem*{theorem*}{Theorem}
\newtheorem*{lemma*}{Lemma}

  \newtheorem{lemma}[theorem]{Lemma}
  \newtheorem{proposition}[theorem]{Proposition}
  
  \newtheorem{corollary}[theorem]{Corollary}
  
  \theoremstyle{definition}

  \newtheorem{conjecture}[theorem]{Conjecture}
  \numberwithin{equation}{section}

  \renewcommand{\C}{{\mathbb C}}
  \newcommand{\D}{{\mathbb D}}

  \newcommand{\E}{{\mathsf E}}

  \newcommand{\cW}{{\mathcal{W}}}

\definecolor{apricot}{rgb}{0.98,0.81,0.69}

  \newcommand{\x}{y}

\usepackage{soul}
\usepackage{ulem}
\usepackage{caption}
\DeclareCaptionLabelFormat{cont}{#1~#2\alph{ContinuedFloat}}
  \newcommand{\bbR}{\mathbb{R}}
  \newcommand{\bbC}{\mathbb{C}}

\usetikzlibrary{shapes}
\usetikzlibrary{plotmarks}
  \allowdisplaybreaks

\author{Jacob S. Chrisiansen$^{1,3}$ and Olof Rubin$^{2}$}

\thanks{$^1$ Centre for Mathematical Sciences, Lund University, Box 118, 22100 Lund, Sweden.
E-mail: jacob\_stordal.christiansen@math.lth.se}

\thanks{$^2$ Centre for Mathematical Sciences, Lund University, Box 118, 22100 Lund, Sweden.
E-mail: olof.rubin@math.lth.se}

\thanks{$^3$ Research supported by VR grant 2023-04054 from the Swedish Research Council and in part by DFF research project 1026-00267B from the Independent Research Fund Denmark.}

\title{Chebyshev polynomials related to Jacobi weights}
\begin{document}
	\maketitle
\begin{abstract}
	We investigate Chebyshev polynomials corresponding to Jacobi weights and determine monotonicity properties of their related Widom factors. This complements work by Bernstein from 1930-31 where the asymptotical behavior of the related Chebyshev norms was established. As a part of the proof, we analyze a Bernstein-type inequality for Jacobi polynomials due to Chow et al. Our findings shed new light on the asymptotical uniform bounds of Jacobi polynomials. We also show a relation between weighted Chebyshev polynomials on the unit circle and Jacobi weighted Chebyshev polynomials on $[-1,1]$. This generalizes work by Lachance et al. In order to complete the picture we provide numerical experiments on the remaining cases that our proof does not cover.
\end{abstract}

\medskip	
	
	
\medskip	
	

\section{Introduction}
In an extensive two part analysis of extremal polynomials on the unit interval, found in \cite{bernstein30,bernstein31}, S.\;N.\;Bernstein investigates the asymptotic behavior of the quantity
\begin{equation}
	\inf_{a_1,\dotsc,a_n}\sup_{x\in [-1,1]}w(x)\left|\prod_{k=1}^{n}(x-a_k)\right|
	\label{eq:chebyshev_norm}
\end{equation}
as $n\rightarrow \infty$ for a variety of different conditions on the weight function $w:[-1,1]\rightarrow [0,\infty)$. If $w$ is bounded and strictly positive on a set consisting of at least $n$ points on $[-1,1]$, there exists a unique set of nodes $\{a_1^\ast,\dotsc,a_n^\ast\}$, all situated in $[-1,1]$, such that
\begin{equation}
	\sup_{x\in [-1,1]}w(x)\left|\prod_{k=1}^{n}(x-a_k^\ast)\right| =\inf_{a_1,\dotsc,a_n}\sup_{x\in [-1,1]}w(x)\left|\prod_{k=1}^{n}(x-a_k)\right|,
\end{equation}
see for instance \cite{achieser56, lorentz86, smirnov-lebedev68}. In other words, the infimum of \eqref{eq:chebyshev_norm} can be replaced by a minimum.
%As a technical remark, we note that while $w$ does not need to be continuous -- in which case the maximum is to be replaced by a supremum -- one can consider its upper semicontinuous regularization for which the same minimal nodes correspond, see \cite[Lemma 1]{novello-schiefermayr-zinchenko21}. For such an upper semicontinuous regularization the maximum is attained. 
The polynomial 
\begin{equation}
	T_n^{w}(x) := \prod_{k=1}^{n}(x-a_k^\ast)
	\label{eq:chebyshev_polynomial}
\end{equation}
is the weighted Chebyshev polynomial of degree $n$ corresponding to the weight function $w$ and it is the unique monic polynomial of degree $n$ minimizing \eqref{eq:chebyshev_norm}. Under the assumption that the weight function is continuous, the minimizer of \eqref{eq:chebyshev_norm} can be characterized by the so-called alternation property. That is to say, a monic polynomial $P$ of degree $n$ is equal to $T_n^w$ (and thus minimizes \eqref{eq:chebyshev_norm}) if and only if there exists $n+1$ points $-1\leq x_0<\dotsc<x_{n}\leq 1$ such that
\begin{equation}
	P_n(x_j)w(x_j) = (-1)^{n-j}\max_{x\in [-1,1]}w(x)|P(x)|.
	\label{eq:alternation}
\end{equation}
This result -- which is central to the study of best approximations on the real line -- can be found in many works on approximation theory, see for instance \cite{lorentz86, achieser56, christiansen-simon-zinchenko-I, novello-schiefermayr-zinchenko21}.

 The study of \eqref{eq:chebyshev_norm} dates back to the work of P.\;L.\;Chebyshev %in
 \cite{chebyshev54,chebyshev59} who considered weight functions of the form $w(x) = 1/P(x)$ where $P$ is a polynomial which is strictly positive on the interval $[-1,1]$. Later A.\;A.\;Markov \cite{markov84} extended these results to allow for weight functions of the form $w(x) = 1/\sqrt{P(x)}$, where again $P$ is a polynomial which is assumed to be strictly positive on $[-1,1]$. It should be stressed that both Chebyshev and Markov provided explicit formulas for the minimizer $T_n^{w}$ in these cases in terms of the polynomial $P$. These formulas can be conveniently found in Achiezer's monograph \cite[Appendix A]{achieser56}.

The considerations in \cite{bernstein30,bernstein31} were different from those of Chebyshev and Markov as Bernstein considered the asymptotic behavior of \eqref{eq:chebyshev_norm} as $n\rightarrow \infty$ for a broad family of weights, not restricting himself to continuous ones. In particular, writing $f(n)\sim g(n)$ as $n\to\infty$ if $f(n)/g(n)\to 1$ as $n\to\infty$, he obtained the following result. 
\begin{theorem}[Bernstein \cite{bernstein31}]
	Let $b_k\in [-1,1]$ and $s_k\in \bbR$ for $k=1,\dotsc,m$, and let $w_0:[-1,1]\rightarrow (0,\infty)$ be a Riemann integrable function such that there exists a value $M>1$ for which $1/M<w_0(x)<M$ holds for every $x\in [-1,1]$. Consider the weight function given by 
	\[w(x) = w_0(x)\prod_{k=1}^{m}|x-b_k|^{s_k}.\]
	Under the above conditions, we have that 
	\begin{equation}
		\min_{a_1,\dotsc,a_n}\sup_{x\in [-1,1]}w(x)\left|\prod_{k=1}^{n}(x-a_k)\right|\sim 2^{1-n}\exp\left\{\frac{1}{\pi}\int_{-1}^{1}\frac{\log w(x)}{\sqrt{1-x^2}}dx\right\}
		\label{eq:bernstein_formula}
	\end{equation}
	as $n\rightarrow \infty$.
	\label{thm:bernstein}
\end{theorem}
This result is shown in two steps. Bernstein initially considers weight functions $w_0$ as in Theorem \ref{thm:bernstein} and finds that if $c_1^\ast,\dotsc,c_n^\ast$ are the unique nodes satisfying
\begin{equation}
	\int_{-1}^{1}\left|\prod_{k=1}^{n}(x-c_k^\ast)\right|^2\frac{w_0(x)^2}{\sqrt{1-x^2}}dx = \min_{c_1,\dotsc,c_n}\int_{-1}^{1}\left|\prod_{k=1}^{n}(x-c_k)\right|^2\frac{w_0(x)^2}{\sqrt{1-x^2}}dx,
	\label{eq:orthogonal_polynomials}
\end{equation}
then the weighted expression
\[w_0(x)\prod_{k=1}^{n}(x-c_k^\ast)\]
is asymptotically alternating on $[-1,1]$, see \cite{bernstein30}. As it turns out, this information suffices to apply a result due to de la Vall\'{e}e-Poussin found in \cite{delavallepoussin10} and ensure that 
\begin{equation}
	\max_{x\in [-1,1]}\left|w_0(x)\prod_{k=1}^{n}(x-c_k^\ast)\right|\sim \min_{a_1,\dotsc,a_n}\max_{x\in [-1,1]}w_0(x)\left|\prod_{k=1}^{n}(x-a_k)\right|
	\label{eq:orthogonal_asymptotic_minimal}
\end{equation}
as $n\rightarrow \infty$. Note that 
\[\prod_{k=1}^{n}(x-c_k^\ast)\] is nothing but the monic orthogonal polynomial of degree $n$ corresponding to the weight function $w_0(x)^2/\sqrt{1-x^2}$ on $[-1,1]$. By showing that
\[\max_{x\in [-1,1]}\left|w_0(x)\prod_{k=1}^{n}(x-c_k^\ast)\right|\sim 2^{1-n}\exp\left\{\frac{1}{\pi}\int_{-1}^{1}\frac{\log w_0(x)}{\sqrt{1-x^2}}dx\right\}\]
as $n\rightarrow \infty$, Bernstein obtains \eqref{eq:bernstein_formula} from \eqref{eq:orthogonal_asymptotic_minimal}. Consequently, Theorem \ref{thm:bernstein} is verified in this case. Proceeding from this, Bernstein extends the analysis by allowing for vanishing factors of the form $|x-b_k|^{s_k}$ to be introduced to the weight where $b_k\in [-1,1]$ and $s_k\in \mathbb{R}$, see \cite{bernstein31}. To show this, he uses a clever technique of bounding the factors $|x-b_k|^{s_k}$ from above and below. However, the connection to the maximal deviation of the corresponding orthogonal polynomials as in \eqref{eq:orthogonal_asymptotic_minimal} gets lost. 

To illustrate why this connection becomes more delicate when zeros are added to the weight function, he provides the explicit example of the monic Jacobi polynomials. Following \cite[Chapter IV]{szego75}, we let $P_n^{(\alpha,\beta)}$ denote the classical Jacobi polynomials with parameters $\alpha,\beta >-1.$ These are uniquely defined by the property that $P_n^{(\alpha,\beta)}$ is a polynomial of exact degree $n$ and
\begin{equation}
	\int_{-1}^{1}(1-x)^\alpha(1+x)^\beta P_m^{(\alpha,\beta)}(x)P_n^{(\alpha,\beta)}(x)dx = \frac{2^{\alpha+\beta+1}}{2n+\alpha+\beta+1}\frac{\Gamma(n+\alpha+1)\Gamma(n+\beta+1)}{\Gamma(n+\alpha+\beta+1)n!}\delta_{nm},
\end{equation}
where $\delta_{nm}$ is the Kronecker delta.
% = 0$ if $n\neq m$ and $\delta_{nn} = 1$ for $n,m\in \{0,1\dotsc\}$.   
Since all zeros of the Jacobi polynomials reside in $[-1,1]$ and since 
\[P_n^{(\alpha,\beta)}(x) = 2^{-n}{{2n+\alpha+\beta}\choose{n}}x^n+\mbox{lower order terms},\]
%where $q$ is a polynomial of degree $n-1$ 
we find that
\begin{equation}
	2^nP_n^{(\alpha,\beta)}(x)\Big/{{2n+\alpha+\beta}\choose{n}} = \prod_{k=1}^{n}(x-\cos \psi_k^\ast)
	\label{eq:monic_jacobi}
\end{equation}
for some values $\psi_k^\ast\in [0,\pi]$. Bernstein's analysis provides the following result.
\begin{theorem}[Bernstein \cite{bernstein31}]
	Let $\rho_\alpha = \frac{\alpha}{2}+\frac{1}{4}$ and $\rho_\beta = \frac{\beta}{2}+\frac{1}{4}$,  and let $\psi_k^\ast$ be associated with $P_n^{(\alpha,\beta)}$ as in \eqref{eq:monic_jacobi}. 
	\begin{enumerate}
		\item If 
		\begin{equation}
			0\leq \rho_\alpha\leq 1/2\quad \text{and}\quad 0\leq \rho_\beta\leq 1/2
			\label{eq:condition_jacobi_1}
		\end{equation}
		both hold, then
		\begin{equation}
			\max_{x\in [-1,1]}(1-x)^{\rho_\alpha}(1+x)^{\rho_\beta}\left|\prod_{k=1}^{n}(x-\cos\psi_k^\ast)\right|\sim 2^{1-n-\rho_\alpha-\rho_\beta}
			\label{eq:jacobi_max_asymptotic}
		\end{equation}
		as $n\rightarrow \infty$.
		\item If one of the conditions in \eqref{eq:condition_jacobi_1} fails to hold, then so does \eqref{eq:jacobi_max_asymptotic}.
	\end{enumerate}
	\label{thm:bernstein_jacobi}
\end{theorem}
Note that by replacing $w_0$ with $(1-x)^{\rho_\alpha}(1+x)^{\rho_\beta}$ in \eqref{eq:orthogonal_polynomials}, the corresponding weight function, to which the orthogonal polynomials are associated, is nothing but
\begin{equation}
	\frac{[(1-x)^{\rho_\alpha}(1+x)^{\rho_\beta}]^2}{\sqrt{1-x^2}} = (1-x)^\alpha(1+x)^\beta.
\end{equation}
It is a simple matter to verify that
\begin{equation}
	2^{1-n}\exp\left\{\frac{1}{\pi}\int_{-1}^{1}\frac{\log [(1-x)^{\rho_\alpha}(1+x)^{\rho_\beta}]}{\sqrt{1-x^2}}dx\right\} = 2^{1-n-\rho_\alpha-\rho_\beta},
	\label{eq:asymptotic_jacobi_cheybshev_norm}
\end{equation}
see for instance \cite{christiansen-eichinger-rubin23} where also a translation of the proof of Theorem \ref{thm:bernstein} appears. As such, we see that the monic Jacobi polynomials defined in \eqref{eq:monic_jacobi} are close to minimal with respect to
\begin{equation}
	\min_{a_1,\dotsc,a_n}\max_{x\in [-1,1]}(1-x)^{\rho_\alpha}(1+x)^{\rho_\beta}\left|\prod_{k=1}^{n}(x-a_k)\right|
	\label{eq:chebyshev_norm_jacobi}
\end{equation}
precisely when \eqref{eq:condition_jacobi_1} holds. 

As a partial result, we wish to carry out a detailed analysis of the quantity 
\[\max_{x\in [-1,1]}(1-x)^{\rho_\alpha}(1+x)^{\rho_\beta}\left|\prod_{k=1}^{n}(x-\cos\psi_k^\ast)\right|\]
for the parameters satisfying \eqref{eq:condition_jacobi_1}. By carefully manipulating a Bernstein-type inequality due to Chow et\,al. \cite{chow-gatteschi-wong94}, we will establish the following result.
\begin{theorem}
	Let $\rho_\alpha = \frac{\alpha}{2}+\frac{1}{4}$ and $\rho_\beta = \frac{\beta}{2}+\frac{1}{4}$, and let $\psi_k^\ast$ be as in \eqref{eq:monic_jacobi}. If \eqref{eq:condition_jacobi_1} holds, then
	\begin{equation}
		\max_{x\in [-1,1]}(1-x)^{\rho_\alpha}(1+x)^{\rho_\beta}\left|\prod_{k=1}^{n}(x-\cos\psi_k^\ast)\right|\leq 2^{1-n-\rho_\alpha-\rho_\beta}
		\label{eq:jacobi_inequality}
	\end{equation}
	for every $n$ with equality if and only if $\rho_\alpha,\rho_\beta\in \{0,1/2\}$.
	\label{thm:jacobi_inequality}
\end{theorem}
By combining Theorem \ref{thm:jacobi_inequality} with Theorem \ref{thm:bernstein_jacobi}, we see that 
\begin{align}
	\begin{split}
		\lim_{n\rightarrow \infty }2^n\max_{x\in [-1,1]}(1-x)^{\rho_\alpha}(1+x)^{\rho_\beta}\left|\prod_{k=1}^{n}(x-\cos\psi_k^\ast)\right|\\  = \sup_n2^n\max_{x\in [-1,1]}(1-x)^{\rho_\alpha}2(1+x)^{\rho_\beta}\left|\prod_{k=1}^{n}(x-\cos\psi_k^\ast)\right|= 2^{1-\rho_\alpha-\rho_\beta}
	\end{split}
\end{align}
which shows that the convergence is from below.
The question of obtaining upper bounds for the quantity
\begin{equation}
	\max_{x\in [-1,1]}(1-x)^{a}(1+x)^{b}\left|\prod_{k=1}^{n}P_n^{(\alpha,\beta)}(x)\right|
\end{equation}
for different values of $a,b,\alpha,\beta$ is well studied, see for instance \cite{erdelyi-magnus-nevai94,chow-gatteschi-wong94, koornwinder-kostenko-teschl18} and has applications to e.g. representation theory and the study of Schr\"{o}dinger operators. Our main interest is the applications that such a result has to the corresponding Chebyshev problem \eqref{eq:chebyshev_norm_jacobi}. To state our main result, we introduce the Widom factor $\cW_n(\rho_\alpha,\rho_\beta)$ defined by
\begin{equation}
	\cW_n(\rho_\alpha,\rho_\beta):= 2^{n}\min_{a_1,\dotsc,a_n}\max_{x\in [-1,1]}(1-x)^{\rho_\alpha}(1+x)^{\rho_\beta}\left|\prod_{k=1}^{n}(x-a_k)\right|.
	\label{eq:widom_factor}
\end{equation}
The choice of naming in honour of H. Widom -- who in \cite{widom69} considered Chebyshev polynomials corresponding to a plethora of compact sets in the complex plane -- stems from \cite{goncharov-hatinoglu15}. The considerations of Widom factors in the literature are plentiful, see for instance \cite{christiansen-simon-zinchenko-I,christiansen-simon-zinchenko-II,christiansen-simon-zinchenko-III,christiansen-simon-zinchenko-IV,alpan22, alpan-zinchenko20,goncharov-hatinoglu15,christiansen-simon-zinchenko-review, novello-schiefermayr-zinchenko21, alpan-zinchenko24}. In this article we wish to carry out a detailed study of the behavior of $\cW_n(\rho_\alpha,\rho_\beta)$ for different parameter values.
\begin{theorem}
	For any value of the parameters $\rho_\alpha,\rho_\beta \geq 0$, we have that
		\[\cW_n(\rho_\alpha,\rho_\beta)\sim 2^{1-\rho_\alpha-\rho_\beta}\]
		as $n\rightarrow \infty$. Furthermore:
		\begin{enumerate}
			\item If $\rho_\alpha,\rho_\beta \in  \{0,1/2\}$, then the quantity $\cW_n(\rho_\alpha,\rho_\beta)$ is constant.
			\item If $\rho_\alpha,\rho_\beta \in [0,1/2]$, then
			\begin{equation}
				\sup_{n}\cW_n(\rho_\alpha,\rho_\beta)= 2^{1-\rho_\alpha-\rho_\beta}.
			\end{equation}
			\item If $\rho_\alpha,\rho_\beta \in \{0\}\cup [1/2,\infty)$, then
			\begin{equation}
				\inf_{n}\cW_n(\rho_\alpha,\rho_\beta) = 2^{1-\rho_\alpha-\rho_\beta},
			\end{equation}
			\begin{equation}
				\sup_n \cW_n(\rho_\alpha,\rho_\beta)\leq \left(\frac{2\rho_\alpha}{\rho_\alpha+\rho_\beta}\right)^{\rho_\alpha}\left(\frac{2\rho_\beta}{\rho_\alpha+\rho_\beta}\right)^{\rho_\beta},
			\end{equation}
			and $\cW_n(\rho_\alpha,\rho_\beta)$ decreases monotonically as $n$ increases.
		\end{enumerate}
		\label{main:thm}
\end{theorem}
The case of $\rho_\alpha,\rho_\beta \in  \{0,1/2\}$ is classical and if $\rho_\beta\in \{0,1/2\}$ and $\rho_\alpha\geq 1/2$ (or vice versa), the result was settled in \cite{bergman-rubin24}. We will in fact apply a similar technique to what is used there for the general setting of $\rho_\alpha,\rho_\beta\geq1/2$. This method originates with \cite{lachance-saff-varga79}. As such, the proofs of the different cases of Theorem \ref{main:thm} uses essentially different techniques and will be split into three separate sections.

\begin{figure}
\centering
\begin{tikzpicture}[scale=3.5]
		\draw[->] (-.5,0) -- (2,0) node[right] {$\rho_\alpha$}; 
	    \draw[->] (0,-.5) -- (0,2) node[above] {$\rho_\beta$};
	\node (r0) at ( 0,  0) {};
   \node (s0) at (0, 0.5) {}; 
   \node (s1) at ( 0.5, 0.5) {}; 
   \node (s2) at ( 0.5, 0) {}; 
  	\draw (0,0.5) node[left] {1/2};
  	 \draw (0.5,0) node[below] {1/2};
   \node (s3) at ( 0.5, 2) {}; 
	   \node (s4) at ( 2, 2) {}; 
	   	   \node (s5) at ( 2, .5) {}; 
   % DRAW TREE
   \fill[fill=gray] (r0.center)--(s0.center)--(s1.center)--(s2.center);
   \path[draw] (r0.center)--(s0.center);
   \path[draw] (s0.center)--(s1.center);
   \path[draw] (s1.center)--(s2.center);
   \path[draw] (s2.center)--(r0.center);

	\fill[fill=gray] (s1.center)--(s3.center)--(s4.center)--(s5.center);
	   \path[draw] (s1.center)--(s5.center);
   \path[draw] (s1.center)--(s3.center);
	\draw[color=black, fill=white] (r0) circle (.02);
   \draw[color=black, fill=white]  (s0) circle (.02);
   \draw[color=black, fill=white]  (s1) circle (.02);
      \draw[color=black, fill=white]  (s2) circle (.02);
    	\end{tikzpicture}
					\caption{The decomposition of the parameter domain as suggested by Theorem \ref{main:thm}. If $\rho_\alpha,\rho_\beta\in [0,1/2]$ (corresponding to the lower left square), $\cW_n(\rho_\alpha,\rho_\beta)\leq 2^{1-\rho_\alpha-\rho_\beta}$ with equality if and only if $\rho_\alpha,\rho_\beta\in \{0,1/2\}$. If $\rho_\alpha,\rho_\beta\in [1/2,\infty)$, the quantity $\cW_n(\rho_\alpha,\rho_\beta)$ decreases monotonically to $2^{1-\rho_\alpha-\rho_\beta}$.}
					\label{fig:parameter_domain}

\end{figure}
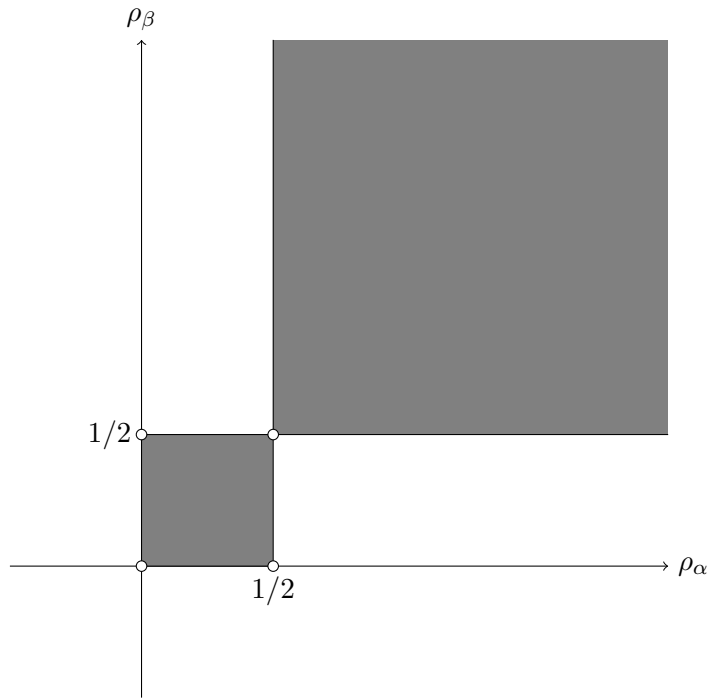

	 We should also note that while negative parameters $\rho_\alpha$ and $\rho_\beta$ are allowed in Theorem \ref{thm:bernstein}, they simply result in the fact that a zero of the minimizer is forcedly placed to cancel the effect of the pole and hence the setting transfers to the one of non-negative parameters if the degree of the associated Chebyshev polynomial is sufficiently large.

Asymptotic results for the Chebyshev polynomials corresponding to \eqref{eq:chebyshev_norm} exist for smooth enough weight functions $w$ which are strictly positive on $[-1,1]$. See for instance \cite{lubinsky-saff87,kroo-peherstorfer08, kroo14}. If the weight function $w$ vanishes somewhere on $[-1,1]$, there does not seem to be many general results for the corresponding Chebyshev polynomials apart from Theorem \ref{thm:bernstein} found in \cite{bernstein31}.

The extension of Chebyshev polynomials to the complex plane was first carried out by Faber in \cite{faber19}. Apart from the proof of Theorem \ref{thm:main_result}, our considerations exclusively concern Chebyshev polynomials corresponding to Jacobi weights and we therefore refrain from a detailed presentation of the general complex case. We note, however, that if $\E\subset \bbC$ is a compact set and $w:\E\rightarrow [0,\infty)$ is a weight function which is non-zero for at least $n$ points, then there exists a unique monic polynomial $T_n^w$ of degree $n$ minimizing
\[\max_{z\in \E}w(z)\left|\prod_{k=1}^{n}(z-a_k)\right|.\]
We invite the reader to consult \cite{smirnov-lebedev68,novello-schiefermayr-zinchenko21} for a detailed presentation of complex Chebyshev polynomials.

Our motivation in providing a detailed study of $\cW_n(\rho_\alpha,\rho_\beta)$ originates from an inquiry of complex Chebyshev polynomials corresponding to star graphs of the form $\{z:z^m\in [-2,2]\}$. It can be shown that these complex Chebyshev polynomials can be directly related to Chebyshev polynomials on $[-1,1]$ with Jacobi weights, see \cite{christiansen-eichinger-rubin23} for details. In the study of the Widom factors for star graphs, certain monotonicity properties appear which can be explained by Theorem \ref{main:thm}.  We trust that Theorem \ref{main:thm} will also be useful in many other connections.

\section{The case of $\rho_\alpha,\rho_\beta \in \{0,1/2\}$}
	We proceed by delving into the proof of Theorem \ref{main:thm}, split into three cases. The first case is very simple. When $\rho_\alpha,\rho_\beta\in \{0,1/2\}$, the minimizers of \eqref{eq:widom_factor} are simply the Chebyshev polynomials of the 1$^{\text{st}}$ to 4$^{\text{th}}$ kind, see e.g. \cite{mason93}. Using basic trigonometry, it follows that if $\theta\in [0,\pi]$ then
		\begin{align}
			\cos \frac{\theta}{2} & = \sqrt{\frac{1+\cos\theta}{2}}, \label{eq:trig-alpha-weight}\\
			\sin \frac{\theta}{2} & = \sqrt{\frac{1-\cos\theta}{2}}.\label{eq:trig-beta-weight}
		\end{align}
		Through the change of variables $x = \cos \theta$, the trigonometric functions
		\begin{align*}
			T_n(x) = 2^{1-n}\cos n\theta,& \quad U_n(x) = 2^{-n}\frac{\sin(n+1)\theta}{\sin \theta},	\\
			V_n(x) = 2^{-n}\frac{\cos\left(n+\frac{1}{2}\right)\theta}{\cos \frac{\theta}{2}}, & \quad W_n(x) = 2^{-n}\frac{\sin\left(n+\frac{1}{2}\right)\theta}{\sin \frac{\theta}{2}},
		\end{align*}
		all define monic polynomials in $x$ of degree $n$. These are precisely, in order, the monic Chebyshev polynomials of the 1$^{\text{st}}$, 2$^{\text{nd}}$, 3$^{\text{rd}}$ and 4$^{\text{th}}$ kind. It is a straightforward matter to verify that they all satisfy the alternating property from \eqref{eq:alternation} upon multiplication with the weight functions
		\begin{align*}
			1,\quad &\sin\theta=\sqrt{1-x^2}, \\
			\sqrt{2}\cos\frac{\theta}{2}=\sqrt{1+x},\quad &\sqrt{2}\sin\frac{\theta}{2}=\sqrt{1-x}, 
		\end{align*}
		respectively. As a result, $T_n, U_n, V_n$, and $W_n$
%the aforementioned ``Alternation theorem'' ensures us that these 
are the minimal configurations sought after in \eqref{eq:widom_factor} for the weight parameters $\rho_\alpha,\rho_\beta\in \{0,1/2\}$. It is evident that $\cW_n(\rho_\alpha,\rho_\beta) = 2^{1-\rho_\alpha-\rho_\beta}$ in all these cases and this shows the first part of Theorem \ref{main:thm}.
		
	\section{The case of $0\leq \rho_\alpha,\rho_\beta \leq 1/2$}	 
	Having settled the first case of Theorem \ref{main:thm}, we turn toward the verification of the second case. For parameter values $\rho_\alpha,\rho_\beta\in [0,1/2]$, we want to show that
	\begin{equation}
		\sup_n\cW_n(\rho_\alpha,\rho_\beta) = 2^{1-\rho_\alpha-\rho_\beta}.
		\label{eq:sup_bound}
	\end{equation}
	These parameter values correspond to the lower left square in Figure \ref{fig:parameter_domain}. Observe that if $\rho_\alpha,\rho_\beta\in \{0,1/2\}$ then this follows from the first case of Theorem \ref{main:thm} since then the sequence $\{\cW_n(\rho_\alpha,\rho_\beta)\}$ is constant. Our approach will consist of first verifying Theorem \ref{thm:jacobi_inequality}. For this reason we remind the reader of our definition of $\psi_k^\ast\in [0,\pi]$ through the formula \begin{equation}
		\prod_{k=1}^{n}(x-\cos\psi_k^\ast)=2^nP_n^{(\alpha,\beta)}(x)\Big/{{2n+\alpha+\beta}\choose{n}},
		\label{eq:monic_jacobi 2}
	\end{equation}
	where  $P_n^{(\alpha,\beta)}$ denotes the Jacobi polynomial with parameters $\alpha,\beta$ and the angles $\psi_k^\ast$ are those for which $\cos\psi_1^\ast,\dotsc,\cos\psi_n^\ast$ enumerate its zeros. We also recall the relation between the $L^2$ and $L^\infty$ parameters
	\begin{equation}
		\begin{cases}
			\rho_\alpha=\frac{\alpha}{2}+\frac{1}{4},\\
		\rho_\beta =\frac{\beta}{2}+\frac{1}{4}.
		\end{cases}
		\label{eq:parameter_relation}
	\end{equation}
	If we can show the validity of \eqref{eq:jacobi_inequality}, then we trivially obtain
	\[\cW_n(\rho_\alpha,\rho_\beta)\leq 2^n\max_{x\in [-1,1]}(1-x)^{\rho_\alpha}(1+x)^{\rho_\beta}\left|\prod_{k=1}^{n}(x-\cos\psi_k^\ast)\right|\leq 2^{1-\rho_\alpha-\rho_\beta}\]
	from which \eqref{eq:sup_bound}, the second case of Theorem \ref{main:thm} follows from the fact that $\cW_n(\rho_\alpha,\rho_\beta)\sim 2^{1-\rho_\alpha-\rho_\beta}$ as $n\rightarrow \infty$. The key to verifying \eqref{eq:jacobi_inequality} is the following result.
	\begin{theorem}[\cite{chow-gatteschi-wong94}]
		If $-1/2\leq \alpha,\beta \leq 1/2$ and $\theta\in[0,\pi]$, then 
		\begin{equation}
			\left(\sin\frac{\theta}{2}\right)^{\alpha+1/2}\left(\cos\frac{\theta}{2}\right)^{\beta+1/2}\left|P_n^{(\alpha,\beta)}(\cos \theta)\right|\leq \frac{\Gamma(q+1)}{\Gamma(\frac{1}{2})}{{n+q}\choose{n}}\left(n+\frac{\alpha+\beta+1}{2}\right)^{-q-\frac{1}{2}}
			\label{eq:chow-gatteschi-wong}
		\end{equation}
		where $q = \max(\alpha,\beta)$.
		\label{thm:chow-gatteschi-wong}
	\end{theorem}
	This result provides a Bernstein-type inequality for Jacobi polynomials and sharpens a previous result of Baratella \cite{baratella86} who had shown \eqref{eq:chow-gatteschi-wong} with the larger factor $2.821$ in place of $\Gamma(q+1)/\Gamma(1/2)$. In the case where $|\alpha| = |\beta| = 1/2$, the inequality in \eqref{eq:chow-gatteschi-wong} is saturated. For arbitrary $-1/2\leq \alpha,\beta \leq 1/2$ the constant $\Gamma(q+1)/\Gamma(1/2)$ is asymptotically sharp as can be seen from an argument using Stirling's formula which we will provide in the proof of Lemma \ref{lem:jacobi_estimate}. The type of inequality exemplified in \eqref{eq:chow-gatteschi-wong} originate with Bernstein who showed in \cite{bernstein31} that the Legendre polynomials $P_n := P_n^{(0,0)}$ satisfy
	\begin{equation}
		\left(\sin \theta\right)^{1/2}|P_n(\cos\theta)|\leq \left(\frac{2}{\pi}\right)^{1/2}n^{-\frac{1}{2}},\quad \theta \in [0,\pi].
		\label{eq:bernstein_ineq}
	\end{equation}
	Antonov and Hol\v{s}hevnikov \cite{antonov-holshevnikov80} sharpened \eqref{eq:bernstein_ineq} and later Lorch \cite{lorch84} extended such an inequality to Gegenbauer polynomials. In terms of Jacobi polynomials, this provides a Bernstein-type inequality for $P_n^{(\lambda,\lambda)}$ with $-1/2\leq \lambda \leq 1/2$. Theorem \ref{thm:chow-gatteschi-wong} contains all these inequalities as special cases. In a different direction, the Erd\'{e}lyi--Magnus--Nevai Conjecture \cite{erdelyi-magnus-nevai94} proposes that if $\hat{P}_n^{(\alpha,\beta)}$ denotes the orthonormal Jacobi polynomial then for any parameter values $\alpha,\beta\geq-1/2$,
	\[\max_{x\in [-1,1]}(1-x)^{\alpha+1/2}(1+x)^{\beta+1/2}(\hat{P}_n^{(\alpha,\beta)}(x))^2 = O(\max[1,(\alpha^2+\beta^2)^{1/4}]).\]
	Theorem \ref{thm:chow-gatteschi-wong} can be used to verify this conjecture in the case where $|\alpha|,|\beta|\leq 1/2$. This is shown in \cite{gautschi09} where an extensive consideration concerning the sharpness of Theorem \ref{thm:chow-gatteschi-wong} is studied numerically for different parameter values.
	
	We consider transferring the trigonometric weights to algebraic weights as in the statement of Theorem \ref{thm:jacobi_inequality}. Letting $x=\cos\theta$, we find that
	\begin{align}
		\left(\sin\frac{\theta}{2}\right)^{\alpha+1/2}&=\left(\frac{1-x}{2}\right)^{\frac{\alpha}{2}+\frac{1}{4}} = 2^{-\rho_\alpha}(1-x)^{\rho_\alpha},\label{eq:sin_weight}\\
		\left(\cos\frac{\theta}{2}\right)^{\beta+1/2}&=\left(\frac{1+x}{2}\right)^{\frac{\beta}{2}+\frac{1}{4}}=2^{-\rho_\beta}(1+x)^{\rho_\beta}, \label{eq:cos_weight}
	\end{align}
	with the parameter relation stated in \eqref{eq:parameter_relation}.
%	We further remind the reader of our definition of $\psi_k^\ast\in [0,\pi]$ through the formula \begin{equation}
%		\prod_{k=1}^{n}(x-\cos\psi_k^\ast)=2^nP_n^{(\alpha,\beta)}(x)\Big/{{2n+\alpha+\beta}\choose{n}},
%		\label{eq:monic_jacobi 2}
%	\end{equation}
%	where the angles $\psi_k^\ast$ are those for which $\cos\psi_1^\ast,\dotsc,\cos\psi_n^\ast$ enumerate the zeros of the Jacobi polynomials. 
	Together with \eqref{eq:sin_weight}, \eqref{eq:cos_weight} and \eqref{eq:monic_jacobi 2}, the change of variables transforms \eqref{eq:chow-gatteschi-wong} into the statement that for $x\in [-1,1]$,
	\begin{equation}
		(1-x)^{\rho_\alpha}(1+x)^{\rho_\beta}\left|\prod_{k=1}^{n}(x-\cos\psi_k^\ast)\right|\leq \frac{\Gamma(q+1)2^{n+\frac{\alpha+\beta+1}{2}}{{n+q}\choose{n}}}{\sqrt{\pi}\left(n+\frac{\alpha+\beta+1}{2}\right)^{q+\frac{1}{2}}{{2n+\alpha+\beta}\choose{n}}},
		\label{eq:gautschi}
	\end{equation}
	where $q = \max(\alpha,\beta)$. From \eqref{eq:parameter_relation} we see that this implies that 
	\begin{equation}
		\cW_n\left(\rho_\alpha,\rho_\beta\right)\leq 2^n\frac{\Gamma(q+1)2^{n+\frac{\alpha+\beta+1}{2}}{{n+q}\choose{n}}}{\sqrt{\pi}\left(n+\frac{\alpha+\beta+1}{2}\right)^{q+\frac{1}{2}}{{2n+\alpha+\beta}\choose{n}}}
		\label{eq:widom_upper_bound_lower_square}
	\end{equation}
	for $\rho_\alpha,\rho_\beta\in[0,1/2]$. In order to determine the upper bound of \eqref{eq:gautschi} and \eqref{eq:widom_upper_bound_lower_square}, we are ultimately led to consider the quantity
	\begin{equation}
		M_n(\alpha,\beta):=\frac{\Gamma(q+1)2^{2n+\frac{\alpha+\beta+1}{2}}{{n+q}\choose{n}}}{\sqrt{\pi}\left(n+\frac{\alpha+\beta+1}{2}\right)^{q+\frac{1}{2}}{{2n+\alpha+\beta}\choose{n}}}.
		\label{eq:M_n}
	\end{equation}
	
	\begin{lemma}
		Suppose that $\alpha,\beta \in [-1/2,1/2]$. Then $M_n(\alpha,\beta)$, defined in \eqref{eq:M_n}, increases monotonically to $2^{1-\rho_\alpha-\rho_\beta}$ as $n\rightarrow\infty$.
		\label{lem:jacobi_estimate}
	\end{lemma}
	The proof we provide constitutes a lengthy computation. We have therefore chosen to illustrate the simpler case where $\alpha = \beta = q$ in the proof below while postponing most of the remaining computations to the appendix. The asymptotical behavior $M_n(\alpha,\beta)\sim 2^{1-\rho_\alpha-\rho_\beta}$ as $n\rightarrow \infty$ merits attention since this shows that the constant $\Gamma(q+1)/\Gamma(1/2)$ is sharp in \eqref{eq:chow-gatteschi-wong}. To see this, note that
	\[\cW_n(\rho_\alpha,\rho_\beta)\leq 2^n\max_{x\in [-1,1]}(1-x)^{\rho_\alpha}(1+x)^{\rho_\beta}\left|\prod_{k=1}^{n}(x-\cos\psi_k^\ast)\right|\leq M_n(\alpha,\beta).\]
	The left-hand side converges to $2^{1-\rho_\alpha-\rho_\beta}$ as $n\rightarrow \infty$ and consequently
				\[\liminf_{n\rightarrow \infty}M_n(\alpha,\beta)\geq 2^{1-\rho_\alpha-\rho_\beta}.\]
				 On the other hand, as Lemma \ref{lem:jacobi_estimate} shows, $M_n(\alpha,\beta)\sim 2^{1-\rho_\alpha-\rho_\beta}$ as $n\rightarrow \infty$ and therefore the constant factor in \eqref{eq:chow-gatteschi-wong} cannot be improved.
	
			\begin{proof}[Proof of Lemma \ref{lem:jacobi_estimate}]
				We rewrite the binomial coefficients using Gamma functions so that
				\begin{equation}
					{{n+q}\choose{n}}=\frac{\Gamma(n+q+1)}{\Gamma(n+1)\Gamma(q+1)},
					\label{eq:binom_1}
				\end{equation}
				\begin{equation}
					{{2n+\alpha+\beta}\choose{n}}=\frac{\Gamma(2n+\alpha+\beta+1)}{\Gamma(n+1)\Gamma(n+\alpha+\beta+1)}.
					\label{eq:binom_2}
				\end{equation}
				The Legendre duplication formula implies that
				\begin{equation}
					\Gamma(2n+\alpha+\beta+1)=\frac{2^{2n+\alpha+\beta}}{\sqrt{\pi}}\Gamma\left(n+\frac{\alpha+\beta+1}{2}\right)\Gamma\left(n+\frac{\alpha+\beta}{2}+1\right).
					\label{eq:legendre}
				\end{equation}
				The combination of \eqref{eq:M_n}, \eqref{eq:binom_1}, \eqref{eq:binom_2} and \eqref{eq:legendre} yields
				\begin{align*}
					M_n(\alpha,\beta)&=\frac{2^{2n+\frac{\alpha+\beta+1}{2}}\Gamma(n+q+1)\Gamma(n+\alpha+\beta+1)}{\sqrt{\pi}(n+\frac{\alpha+\beta+1}{2})^{q+1/2}\Gamma(2n+\alpha+\beta+1)}\\
					&=2^{\frac{1-\alpha-\beta}{2}}\frac{\Gamma(n+q+1)\Gamma(n+\alpha+\beta+1)}{(n+\frac{\alpha+\beta+1}{2})^{q+\frac{1}{2}}\Gamma(n+\frac{\alpha+\beta+1}{2})\Gamma(n+\frac{\alpha+\beta}{2}+1)}.
				\end{align*}
				The Gamma function satisfies that $\Gamma(n+a)/\Gamma(n+b) \sim n^{a-b}$ as $n\rightarrow \infty$, and we may conclude that
				\begin{align*}
					M_n(\alpha,\beta)\sim2^{\frac{1-\alpha-\beta}{2}}\frac{n^{q+1}n^{\alpha+\beta+1}}{n^{q+\frac{1}{2}}n^{\frac{\alpha+\beta+1}{2}}n^{\frac{\alpha+\beta}{2}+1}}=2^{\frac{1-\alpha-\beta}{2}}=2^{1-\rho_\alpha-\rho_\beta}	
				\end{align*}
				as $n\rightarrow\infty$. 				
				  
				We are left to determine the monotonicity of $M_n(\alpha,\beta)$. For this reason we consider the quotient 
				\begin{align}
				\begin{split}
					\frac{M_{n+1}(\alpha,\beta)}{M_n(\alpha,\beta)}&=\frac{(n+q+1)(n+\alpha+\beta+1)\left(n+\frac{\alpha+\beta+1}{2}\right)^{q+\frac{1}{2}}}{\left(n+\frac{\alpha+\beta+1}{2}+1\right)^{q+\frac{1}{2}}(n+\frac{\alpha+\beta+1}{2})(n+\frac{\alpha+\beta}{2}+1)}\\
					&=\frac{(n+q+1)(n+\alpha+\beta+1)\left(n+\frac{\alpha+\beta+1}{2}\right)^{q-\frac{1}{2}}}{\left(n+\frac{\alpha+\beta+1}{2}+1\right)^{q+\frac{1}{2}}(n+\frac{\alpha+\beta}{2}+1)}.
				\end{split}
				\label{eq:M_n_quot}
					\end{align}
				Whether this quotient is bigger than or smaller than $1$ determines the monotonicity of the sequence $\{M_n(\alpha,\beta)\}$. To simplify, we begin by considering the case $\alpha=\beta = q$. Writing $M_n(q):=M_n(q,q)$, the quotient \eqref{eq:M_n_quot} becomes
				\begin{align*}
					\frac{M_{n+1}(q)}{M_n(q)} & = \frac{(n+1+q)(n+2q+1)}{(n+q+\frac{1}{2})(n+q+1)}\left(\frac{n+q+\frac{1}{2}}{n+q+\frac{3}{2}}\right)^{q+\frac{1}{2}}\\
					& = (n+2q+1)\frac{(n+q+\frac{1}{2})^{q-\frac{1}{2}}}{(n+q+\frac{3}{2})^{q+\frac{1}{2}}}.
				\end{align*}
				We introduce the function
				\[f(x) = (x+2q+1)\frac{(x+q+\frac{1}{2})^{q-\frac{1}{2}}}{(x+q+\frac{3}{2})^{q+\frac{1}{2}}},\]
				and claim that $f'(x)<0$ for $x>0$. From the fact that $f(x)\rightarrow 1$ as $x\rightarrow \infty$, this will imply that $f(x)\geq 1$ for every $x>0$. As a consequence we see that $M_{n+1}(q)\geq f(n)M_n(q)\geq M_n(q)$.
				
				Taking the logarithmic derivative of $f$ generates
				\[\frac{f'(x)}{f(x)} = \frac{q^2-\frac{1}{4}}{(x+2q+1)(x+q+\frac{1}{2})(x+q+\frac{3}{2})}.\]
%				\begin{align*}
%					\frac{f'(x)}{f(x)} & = \frac{1}{x+2q+1}+\frac{q-\frac{1}{2}}{x+q+\frac{1}{2}}-\frac{q+\frac{1}{2}}{x+q+\frac{3}{2}} \\
%					& = \frac{(x+q+\frac{1}{2})(x+q+\frac{3}{2})+(x+2q+1)\Big[(q-\frac{1}{2})(x+q+\frac{3}{2})-(q+\frac{1}{2})(x+q+\frac{1}{2}))\Big]}{(x+2q+1)(x+q+\frac{1}{2})(x+q+\frac{3}{2})}\\
%					& = \frac{x^2+2x(q+1)+(q+\frac{1}{2})(q+\frac{3}{2})+(x+2q+1)\Big[-x+(q-\frac{1}{2})(q+\frac{3}{2})-(q+\frac{1}{2})^2\Big]}{(x+2q+1)(x+q+\frac{1}{2})(x+q+\frac{3}{2})}\\
%					& = \frac{x^2+2x(q+1)+(q+\frac{1}{2})(q+\frac{3}{2})-(x+2q+1)(x+1)}{(x+2q+1)(x+q+\frac{1}{2})(x+q+\frac{3}{2})}\\
%					& = \frac{q^2-\frac{1}{4}}{(x+2q+1)(x+q+\frac{1}{2})(x+q+\frac{3}{2})}.
%				\end{align*}
				If $-1/2<q<1/2$, then it is clear that 
				\[\frac{f'(x)}{f(x)}<0\]
				and as a consequence $f'(x)<0$. We conclude that $M_{n+1}(q)\geq M_n(q)$ and this settles the case where $\alpha= \beta = q$. 
				
				If $\alpha\neq \beta$, then the problem is more delicate. We again need to show that the quotient \eqref{eq:M_n_quot} is strictly greater than $1$. In order to show this, we reintroduce the function $f$ by defining it as
			\begin{align}
				f(x)=\frac{(x+q+1)(x+\alpha+\beta+1)}{(x+\frac{\alpha+\beta}{2}+1)}\frac{(x+\frac{\alpha+\beta+1}{2})^{q-1/2}}{(x+\frac{\alpha+\beta+1}{2}+1)^{q+1/2}}.
				\label{eq:thefunctionf}
			\end{align}
			In a completely analogous manner, we only need to show that $f'(x)<0$ for every $x\in (0,\infty)$ in order to conclude that $f(x)\geq 1$ and consequently $M_{n+1}(\alpha,\beta)\geq f(n)M_{n}(\alpha,\beta)\geq M_n(\alpha,\beta)$. Since $M_n$ is symmetric with respect to its variables and $q = \max(\alpha,\beta)$, it is enough to consider the parameter values $-1/2\leq \beta\leq \alpha\leq 1/2$.
			 
			The logarithmic derivative of $f$ is given by
			\[\frac{f'(x)}{f(x)}=\frac{1}{x+q+1}+\frac{1}{x+\alpha+\beta+1}+\frac{q-1/2}{x+\frac{\alpha+\beta+1}{2}}-\frac{1}{x+\frac{\alpha+\beta}{2}+1}-\frac{q+1/2}{x+\frac{\alpha+\beta+1}{2}+1}.\]
			Written over common denominator and utilizing that $\alpha = q$, this expression becomes
			\begin{equation}
				\frac{f'(x)}{f(x)}=\frac{c_2(\alpha,\beta)x^2+c_1(\alpha,\beta)x+c_0(\alpha,\beta)}{(x+\alpha+1)(x+\alpha+\beta+1)(x+\frac{\alpha+\beta}{2})(x+\frac{\alpha+\beta}{2}+1)(x+\frac{\alpha+\beta}{2}+\frac{3}{2})},
				\label{eq:logarithmic_derivative}
			\end{equation}
			where
			\begin{align}
			c_2(\alpha,\beta)&=\frac{\alpha^2}{2}+\frac{\beta^2}{2}-\frac{1}{4},\label{eq:c_2}\\
			c_1(\alpha,\beta)&=\frac{3\alpha^3}{4}+\frac{(4\beta+8)\alpha^2}{8}+\frac{(\beta^2-1)\alpha}{4}+\frac{\beta^3}{2}+\beta^2-\frac{\beta}{4}-\frac{1}{2},\label{eq:c_1}\\
			c_0(\alpha,\beta)&=\frac{\alpha^4}{4}+\frac{(3\beta+6)\alpha^3}{8}+\frac{(\beta+2)^2\alpha^2}{8}+\frac{(\beta^2-1)(\beta+2)\alpha}{8}+\frac{\beta^4}{8}+\frac{\beta^3}{2}+\frac{3\beta^2}{8}-\frac{\beta}{4}-\frac{1}{4}.\label{eq:c_0}
			\end{align}
			Now it is immediately clear that $c_2(\alpha,\beta)<0$ unless $|\alpha|=|\beta|=1/2$. We claim that the same is true for the remaining coefficients $c_1(\alpha,\beta)$ and $c_0(\alpha,\beta)$. Namely, $c_k(\alpha,\beta)<0$ for $k=0,1,2$ unless $|\alpha|=|\beta| = 1/2$. A consequence of these inequalities is that $f'(x)<0$ for $x>0$ which is precisely the sought after inequality. The proof of this is provided in Lemma \ref{lem:coeff_analysis} in the appendix. As a result, we obtain that $f(n)\geq 1$ and therefore
			\[M_{n+1}(\alpha,\beta)=f(n)M_n(\alpha,\beta)\geq M_n(\alpha,\beta).\]
			In conclusion, $M_n(\alpha,\beta)$ converges monotonically to $2^{1-\rho_\alpha-\rho_\beta}$ from below. \hfill\qedhere
			\end{proof}
			
			An immediate consequence of Lemma \ref{lem:jacobi_estimate} together with \eqref{eq:gautschi} is that 
			\begin{equation}
				\cW_n(\rho_\alpha,\rho_\beta)\leq 2^n\max_{x\in [-1,1]}(1-x)^{\rho_\alpha}(1+x)^{\rho_\beta}\left|\prod_{k=1}^{n}(x-\cos\psi_k^\ast)\right|\leq M_n(\alpha,\beta)\leq 2^{1-\rho_\alpha-\rho_\beta}.
			\end{equation}
			This simultaneously shows that Theorem \ref{thm:jacobi_inequality} holds and that 
			\[\cW_n(\rho_\alpha,\rho_\beta)\leq 2^{1-\rho_\alpha-\rho_\beta}\]
			is valid for every $n$. On the other hand, \eqref{eq:asymptotic_jacobi_cheybshev_norm} implies that these quantities are all asymptotically equal as $n\rightarrow \infty$. This is possible only if
			\[\sup_n\cW_n(\rho_\alpha,\rho_\beta)= 2^{1-\rho_\alpha-\rho_\beta}\]
			and hence the second case of Theorem \ref{main:thm} is shown. From the fact that $\cW_{n}(\rho_\alpha,\rho_\beta)\leq M_n(\alpha,\beta)$, we gather the following corollary to Lemma \ref{lem:jacobi_estimate}.
			\begin{corollary}
				If $0\leq \rho_\alpha, \rho_\beta\leq 1/2$ and $r=2\max(\rho_\alpha, \rho_\beta)$, then
				\begin{equation}
					\cW_n(\rho_\alpha,\rho_\beta)\leq \frac{\Gamma(r+1/2)2^{2n+\rho_\alpha+\rho_\beta}{{n+r-1/2}\choose{n}}}{\sqrt{\pi}(n+\rho_\alpha+\rho_\beta)^{r}{{2n+2\rho_\alpha+2\rho_\beta-1}\choose{n}}}\leq 2^{1-\rho_\alpha-\rho_\beta}
				\end{equation}
				and the rightmost inequality is strict unless $\rho_\alpha, \rho_\beta\in \{0,1/2\}$.
			\end{corollary}

\section{The case of $\rho_\alpha,\rho_\beta\geq 1/2$}
	
	We are left to prove the third and final case of Theorem \ref{main:thm}. Our approach to proving this revolves around relating Jacobi weighted Chebyshev polynomials to weighted Chebyshev polynomials on the unit circle. These ideas have previously been investigated in \cite{lachance-saff-varga79, bergman-rubin24} for certain choices of the parameters. We will provide a general argument which works for any choice of parameters $\rho_\alpha,\rho_\beta\geq 1/2$. 
	
	The approach first studied in \cite{lachance-saff-varga79} hinges upon the Erd\H{o}s--Lax inequality, see \cite{lax44}. If $P$ is a polynomial of degree $n$ which does not vanish anywhere on $\D$, then
	\begin{equation}
		\max_{|z| = 1}|P'(z)| \leq \frac{n}{2}\max_{|z| = 1}|P(z)|.
		\label{eq:erdos_lax}
	\end{equation}
	This result relates the maximum modulus of a polynomial with its derivative. Furthermore, equality occurs in \eqref{eq:erdos_lax} if and only if all zeros of $P$ lie on the unit circle. A recent extension of \eqref{eq:erdos_lax} to the case of generalized polynomials having all zeros situated on the unit circle is the following.
	\begin{theorem}[\cite{bergman-rubin24}]
	\label{thm:bergman_rubin}
		Let $s_k\geq 1$ and $\theta_k\in [0,2\pi)$ for $k = 1,\dotsc,n$. Then
		\[\max_{|z| = 1}\left|\frac{d}{dz}\left\{\prod_{k=1}^{n}(z-e^{i\theta_k})^{s_k}\right\}\right| = \frac{\sum_{k=1}^{n}s_k}{2}\max_{|z| = 1}\left|\prod_{k=1}^{n}(z-e^{i\theta_k})^{s_k}\right|.\]
	\end{theorem}
	Given a polynomial with zeros in the unit disk, there is a procedure of constructing a related polynomial whose zeros all lie on the unit circle.
	\begin{lemma}[P\'{o}lya and Szeg\H{o} \cite{polya-szego}]
		Let $|a_k|\leq 1$ for $k = 1,\dotsc,n$ and suppose that $m\in \mathbb{Z_+}$ and $\theta\in \bbR$. Then
		\[z^m\prod_{k=1}^{n}(z-a_k)+e^{i\theta}\prod_{k=1}^{n}(1-\overline{a_k} z)\]
		does not vanish away from $|z| = 1$.
		\label{thm:polya-szego}
	\end{lemma}
	
	Utilizing these two results, we obtain the following theorem which ``opens up'' the interval to the circle.
	
	\begin{theorem}
		Given $\rho_\alpha,\rho_\beta \geq 1/2$, let $\theta_1^\ast,\dotsc,\theta_{n}^\ast\in [0,2\pi)$ be the unique minimizing data of the expression
		\begin{equation}
			I_n:=\min_{\theta_1,\dotsc,\theta_n}\max_{x\in [-1,1]}\left|(1-x)^{\rho_\alpha}(1+x)^{\rho_\beta}\prod_{k=1}^{n}(x-\cos\theta_k)\right|.
			\label{eq:chebyshev_norm_jacobi_2}
		\end{equation}
		In that case, 
		\[\frac{1}{2\rho_\alpha+2\rho_\beta+2n}\frac{d}{dz}\left\{(z-1)^{2\rho_\alpha}(z+1)^{2\rho_\beta}\prod_{k=1}^{n}(z-e^{i\theta_k^\ast})(z-e^{-i\theta_k^\ast})\right\} \]
		is the unique minimizer of 
		\begin{equation}
			C_n:=\min_{a_1,\dotsc,a_{2n+1}}\max_{|z| = 1}\left|(z-1)^{2\rho_\alpha-1}(z+1)^{2\rho_\beta-1}\prod_{k=1}^{2n+1}(z-a_k)\right|.
			\label{eq:unit_circle}
		\end{equation}
		\label{thm:main_result}
		Furthermore, we have that
		\begin{equation}
		   C_n = 2^{n+\rho_\alpha+\rho_\beta-1} I_n.
		\end{equation}
%	\begin{align}
%		\min_{a_1,\dotsc,a_{2n+1}}\max_{|z| = 1}\left|(z-1)^{2\rho_\alpha-1}(z+1)^{2\rho_\beta-1}\prod_{k=1}^{2n+1}(z-a_k)\right| \notag \\
%		=2^{\rho_\alpha+\rho_\beta+n-1}\min_{\theta_1,\dotsc,\theta_n}\max_{x\in [-1,1]}\left|(1-x)^{\rho_\alpha}(1+x)^{\rho_\beta}\prod_{k=1}^{n}(x-\cos\theta_k)\right|.
%	\end{align}
	\end{theorem}
	
	\begin{proof}
		We will prove this result in the opposite direction as it is stated by first analyzing the minimizer of \eqref{eq:unit_circle}. Assume that the points $a_k^\ast$ are uniquely chosen to satisfy
		\[\max_{|z| = 1}\left|(z-1)^{2\rho_\alpha-1}(z+1)^{2\rho_\beta-1}\prod_{k=1}^{2n+1}(z-a_k^\ast)\right| = C_n.
%\min_{a_k}\max_{|z| = 1}\left|(z-1)^{2\rho_\alpha-1}(z+1)^{2\rho_\beta-1}\prod_{k=1}^{2n+1}(z-a_k)\right|.
\]
		A result of Fej\'{e}r \cite{fejer22} says that the zeros of a complex Chebyshev polynomial corresponding to a compact set $\E$ are always situated in the convex hull of $\E$. In our setting, a similar reasoning ensures us that $|a_k^\ast|\leq 1$ holds for $k = 1,\dotsc 2n+1$. It is actually possible to show that this inequality is strict, see \cite{lachance-saff-varga79} for details. We form the combination
		\begin{equation}
			P(z) = z\prod_{k=1}^{2n+1}(z-a_k^\ast)-\prod_{k=1}^{2n+1}(1-\overline{a_k^\ast} z).
			\label{eq:P}
		\end{equation}
		It is immediate that $P$ is a monic polynomial of degree $2n+2$ and by Lemma \ref{thm:polya-szego} we conclude that $P$ has all its zeros on the unit circle. Since the minimizer of \eqref{eq:unit_circle} is unique, all $a_k^\ast$ must come in conjugate pairs with one possible exception. Due to symmetry of the weight function, this exceptional zero must always be real.
		%The exceptional zero is, however, real due to symmetry. 
		Therefore,
		\[P(z) = z\prod_{k=1}^{2n+1}(z-a_k^\ast)-\prod_{k=1}^{2n+1}(1-a_k^\ast z)\]
		and as a consequence,
		\begin{align*}
			P(1) & = \prod_{k=1}^{2n+1}(1-a_k^\ast)-\prod_{k=1}^{2n+1}(1-a_k^\ast) = 0, \\
			P(-1) & = -\prod_{k=1}^{2n+1}(-1-a_k^\ast)-\prod_{k=1}^{2n+1}(1+a_k^\ast) = 0.	
		\end{align*}
		We conclude that
		\begin{equation}
			P(z) = (z-1)(z+1)\prod_{k=1}^{2n}(z-e^{i\theta_k^\ast})
		\end{equation}
		for some values $\theta_k^\ast\in [0,2\pi)$. 
		
		Now let $\varphi_k\in [0,2\pi)$ be chosen such that
		\begin{equation}
			\max_{|z| = 1}\left|(z-1)^{2\rho_\alpha}(z+1)^{2\rho_\beta}\prod_{k=1}^{2n}(z-e^{i\varphi_k})\right| = \min_{\theta_1,\dotsc,\theta_{2n}}\max_{|z| = 1}\left|(z-1)^{2\rho_\alpha}(z+1)^{2\rho_\beta}\prod_{k=1}^{2n}(z-e^{i\theta_k})\right|. 
			\label{eq:minimization_unit_circle}
		\end{equation}
		From the fact that 
		\[\prod_{k=1}^{2n+1}\left|1-a_k^\ast z\right| = \prod_{k=1}^{2n+1}|z-a_k^\ast|\]
		for $|z| = 1$ together with the representation in \eqref{eq:P}, we conclude from the triangle inequality that
		\begin{align}
			\max_{|z| = 1}\left|(z-1)^{2\rho_\alpha}(z+1)^{2\rho_\beta}\prod_{k=1}^{2n}(z-e^{i\varphi_k})\right| \leq \max_{|z| = 1}\left|(z-1)^{2\rho_\alpha-1}(z+1)^{2\rho_\beta-1}P(z)\right|\notag\\
			\leq 2 \max_{|z| = 1}\left|(z-1)^{2\rho_\alpha-1}(z+1)^{2\rho_\beta-1}\prod_{k=1}^{2n+1}(z-a_k^\ast)\right| = 2 C_n.   \label{eq:first_ineq}
			%\min_{a_1,\dotsc,a_{2n+1}}\max_{|z| = 1}\left|(z-1)^{2\rho_\alpha-1}(z+1)^{2\rho_\beta-1}\prod_{k=1}^{2n+1}(z-a_k)\right|.
		\end{align}
		On the other hand, the expression
		\[\frac{1}{2\rho_\alpha+2\rho_\beta+2n}\frac{d}{dz}(z-1)^{2\rho_\alpha}(z+1)^{2\rho_\beta}\prod_{k=1}^{2n}(z-e^{i\varphi_k})\]
		is of the form
		\[(z-1)^{2\rho_\alpha-1}(z+1)^{2\rho_\beta-1}\prod_{k=1}^{2n+1}(z-b_k)\]
		for some values of $b_k\in \C$. From the Gauss--Lucas Theorem we can actually conclude that $|b_k|\leq 1$. As it turns out, this derivative is a candidate for a minimizer of \eqref{eq:unit_circle}. Therefore,
		\begin{align}
			\max_{|z|=1}\left|\frac{1}{2\rho_\alpha+2\rho_\beta+2n}\frac{d}{dz}\left\{(z-1)^{2\rho_\alpha}(z+1)^{2\rho_\beta}\prod_{k=1}^{2n}(z-e^{i\varphi_k})\right\}\right|\notag\\ \geq \max_{|z| = 1}\left|(z-1)^{2\rho_\alpha-1}(z+1)^{2\rho_\beta-1}\prod_{k=1}^{2n+1}(z-a_k^\ast)\right|
			\label{eq:second_ineq}
		\end{align}
		and Theorem \ref{thm:bergman_rubin} immediately gives us that 
		\begin{align}
			2\max_{|z|=1}\left|\frac{1}{2\rho_\alpha+2\rho_\beta+2n}\frac{d}{dz}\left\{(z-1)^{2\rho_\alpha}(z+1)^{2\rho_\beta}\prod_{k=1}^{2n}(z-e^{i\varphi_k})\right\}\right|\notag \\ = \max_{|z| = 1}\left|(z-1)^{2\rho_\alpha}(z+1)^{2\rho_\beta}\prod_{k=1}^{2n}(z-e^{i\varphi_k})\right|.
			\label{eq:third_eq}
		\end{align}
		By combining \eqref{eq:first_ineq}, \eqref{eq:second_ineq} and \eqref{eq:third_eq}, we conclude that equality must hold in all the inequalities. %throughout all equations. 
		In particular,
		\begin{align}
		\max_{|z|=1}\left|\frac{1}{2\rho_\alpha+2\rho_\beta+2n}\frac{d}{dz}\left\{(z-1)^{2\rho_\alpha}(z+1)^{2\rho_\beta}\prod_{k=1}^{2n}(z-e^{i\varphi_k})\right\}\right|\notag\\ = \max_{|z| = 1}\left|(z-1)^{2\rho_\alpha-1}(z+1)^{2\rho_\beta-1}\prod_{k=1}^{2n+1}(z-a_k^\ast)\right|	
		\label{eq:equality_holds_derivative}
		\end{align}
		and 
		\begin{align}
			\max_{|z| = 1}\left|(z-1)^{2\rho_\alpha}(z+1)^{2\rho_\beta}\prod_{k=1}^{2n}(z-e^{i\varphi_k})\right| = \max_{|z| = 1}\left|(z-1)^{2\rho_\alpha-1}(z+1)^{2\rho_\beta-1}P(z)\right|. \label{eq:equality_minimizer_on_circle}
		\end{align}
		As a consequence of \eqref{eq:equality_holds_derivative} together with uniqueness of the minimizer, we have that 
		\[\frac{1}{2\rho_\alpha+2\rho_\beta+2n}\frac{d}{dz}\left\{(z-1)^{2\rho_\alpha}(z+1)^{2\rho_\beta}\prod_{k=1}^{2n}(z-e^{i\varphi_k})\right\} =(z-1)^{2\rho_\alpha-1}(z+1)^{2\rho_\beta-1}\prod_{k=1}^{2n+1}(z-a_k^\ast) \]
		and so the minimizing points $\varphi_k$, which were chosen to satisfy \eqref{eq:minimization_unit_circle}, are determined from the point set $\{a_k^\ast\}$. Incidentally we have established that the solution to \eqref{eq:minimization_unit_circle} is unique and as a consequence all nodes $e^{i\varphi_k}$ come in conjugate pairs. Furthermore, \eqref{eq:equality_minimizer_on_circle} implies that after a possible rearrangement, it holds that $\theta_k^\ast = \varphi_k$ for $k = 1,\dotsc,2n$.
		
		We now let $x= (z+z^{-1})/2$ and recall that $x\in [-1,1]$ when $|z| = 1$. Under this transformation, 
		\[(z-e^{i\theta_k})(z-e^{-i\theta_k}) = 2z(x-\cos\theta_k)\]
		so if we arrange the angles $\theta_k^\ast$ in such a way that $\theta_{k+n}^\ast = \theta_k^\ast+\pi$ for $k=1,\dotsc,n$, then
		\begin{equation}(z-1)^{2\rho_\alpha}(z+1)^{2\rho_\beta}\prod_{k=1}^{2n}(z-e^{i\theta_k\ast}) = (2z)^{\rho_\alpha+\rho_\beta+n}(x-1)^{\rho_\alpha}(x+1)^{\rho_\beta}\prod_{k=1}^{n}(x-\cos\theta_k^\ast).\label{eq:relation_interval}\end{equation}
		In conclusion, to each expression of the form
		\[(x-1)^{\rho_\alpha}(x+1)^{\rho_\beta}\prod_{k=1}^{n}(x-\cos\theta_k),\]
		there corresponds a function
		\[(z-1)^{2\rho_\alpha}(z+1)^{2\rho_\beta}\prod_{k=1}^{n}(z-e^{i\theta_k})(z-e^{-i\theta_k})\]
		and their absolute values are related through multiplication with the constant $2^{\rho_\alpha+\rho_\beta+n}$. Since the left-hand side of \eqref{eq:relation_interval} is minimal in terms of maximal modulus on the unit circle, we conclude that 
		\[\max_{x\in [-1,1]}\left|(x-1)^{\rho_\alpha}(x+1)^{\rho_\beta}\prod_{k=1}^{n}(x-\cos\theta_k^\ast)\right| = I_n. %\min_{\theta_1,\dotsc,\theta_n}\max_{x\in [-1,1]}\left|(x-1)^{\rho_\alpha}(x+1)^{\rho_\beta}\prod_{k=1}^{n}(x-\cos\theta_k)\right|.
		\]
		The result now follows from the uniqueness of the minimizer of \eqref{eq:chebyshev_norm_jacobi_2}.\hfill\qedhere
	\end{proof}
	
	We end this section by listing some consequences of Theorem \ref{thm:main_result}.
	\begin{corollary}
		When $\rho_\alpha,\rho_\beta\geq 1/2$, there is a unique configuration  $\theta_1^\ast,\dotsc,\theta_{2n}^\ast\in [0,2\pi)$ such that 
		\begin{equation}
			\max_{|z| = 1}\left|(z-1)^{2\rho_\alpha}(z+1)^{2\rho_\beta}\prod_{k=1}^{2n}(z-e^{i\theta_k^\ast})\right| = \min_{\theta_1,\dotsc,\theta_{2n}}\max_{|z| = 1}\left|(z-1)^{2\rho_\alpha}(z+1)^{2\rho_\beta}\prod_{k=1}^{2n}(z-e^{i\theta_k})\right|. 
		\end{equation}
	\end{corollary}
	Our main conclusion is the following monotonicity result which constitutes the final puzzle piece to proving the third case in Theorem \ref{main:thm}.
	\begin{corollary}
		If $\rho_\alpha,\rho_\beta \geq 1/2$, then $\cW_n(\rho_\alpha,\rho_\beta)$ decays monotonically to $2^{1-\rho_\alpha-\rho_\beta}$ as $n\rightarrow \infty$ and for every $n$, we have that
		\[\cW_n(\rho_\alpha,\rho_\beta)\leq \left(\frac{2\rho_\alpha}{\rho_\alpha+\rho_\beta}\right)^{\rho_\alpha}\left(\frac{2\rho_\beta}{\rho_\alpha+\rho_\beta}\right)^{\rho_\beta}.\]
	\end{corollary}
	\begin{proof}
		We begin by verifying the monotonicity of the sequence of Widom factors. From \eqref{eq:widom_factor} and Theorem \ref{thm:main_result} we know that
		\begin{align*}
		\cW_n(\rho_\alpha,\rho_\beta) = 
		2^{n}\min_{\theta_1,\dotsc,\theta_n}\max_{x\in [-1,1]}(1-x)^{\rho_\alpha}(1+x)^{\rho_\beta}\left|\prod_{k=1}^{n}(x-\cos\theta_k)\right| \\ =
		2^{1-\rho_\alpha-\rho_\beta}\min_{a_1,\dotsc,a_{2n+1}}\max_{|z| = 1}\left|(z-1)^{2\rho_\alpha-1}(z+1)^{2\rho_\beta-1}\prod_{k=1}^{2n+1}(z-a_k)\right|.		
		\end{align*}
Clearly the right-hand side decreases monotonically in $n$ since multiplying by $z^2$ does not change the absolute value and
		\begin{multline*}
%\max_{|z| = 1}\left|(z-1)^{2\rho_\alpha-1}(z+1)^{2\rho_\beta-1}\prod_{k=1}^{2n+1}(z-a_k^\ast)\right| = 
			\max_{|z|=1}\left|z^2(z-1)^{2\rho_\alpha-1}(z+1)^{2\rho_\beta-1}\prod_{k=1}^{2n+1}(z-a_k^\ast)\right| \\
			\geq \min_{a_1,\dotsc,a_{2(n+1)+1}}\max_{|z| = 1}\left|(z-1)^{2\rho_\alpha-1}(z+1)^{2\rho_\beta-1}\prod_{k=1}^{2(n+1)+1}(z-a_k)\right|.
		\end{multline*}
Hence monotonicity of $\cW_n(\rho_\alpha,\rho_\beta)$ follows.
		
		To show the upper bound, we note that Theorem \ref{thm:main_result} is still valid if $n = 0$ meaning that the expression
	\begin{equation}
		\frac{1}{2\rho_\alpha+2\rho_\beta}\frac{d}{dz}\left\{(z-1)^{2\rho_\alpha}(z+1)^{2\rho_\beta}\right\} = (z-1)^{2\rho_\alpha-1}(z+1)^{2\rho_\beta-1}(z-a^\ast)
	\end{equation}
	minimizes \eqref{eq:unit_circle} in the case where $n=0$. Consequently,
	\begin{multline}
	2^{1-\rho_\alpha-\rho_\beta}\min_{a}\max_{|z|=1}\left|(z-1)^{2\rho_\alpha-1}(z+1)^{2\rho_\beta-1}(z-a)\right|  = \max_{x\in [-1,1]}\left|(1-x)^{\rho_\alpha}(1+x)^{\rho_\beta}\right| \\	
	 = \left(1+\frac{\rho_\alpha-\rho_\beta}{\rho_\alpha+\rho_\beta}\right)^{\rho_\alpha}\left(1-\frac{\rho_\alpha-\rho_\beta}{\rho_\alpha+\rho_\beta}\right)^{\rho_\beta} = \left(\frac{2\rho_\alpha}{\rho_\alpha+\rho_\beta}\right)^{\rho_\alpha}\left(\frac{2\rho_\beta}{\rho_\alpha+\rho_\beta}\right)^{\rho_\beta}
	\end{multline}
	and the upper bound now follows from the monotonicity of $\cW_n(\rho_\alpha,\rho_\beta)$. \hfill\qedhere
	\end{proof}
	
		\section{What about the remaining parameters $\rho_\alpha,\rho_\beta$?}
	To conclude our investigation we question what can be said concerning the monotonicity properties of $\cW_n(\rho_\alpha,\rho_\beta)$ for the remaining values of the parameters.  These are precisely those which satisfy 
	\[0<\rho_\alpha<1/2\,\, \text{ and }\,\,\rho_\beta\geq 1/2\quad \text{or}\quad 0<\rho_\beta<1/2\,\, \text{ and }\,\, \rho_\alpha\geq 1/2,\]
	as illustrated by the blank strips in Figure \ref{fig:speculative_domain}.
	
	\begin{figure}
\centering
\begin{tikzpicture}[scale=3.5]
		\draw[->] (-.5,0) -- (2,0) node[right] {$\rho_\alpha$}; 
	    \draw[->] (0,-.5) -- (0,2) node[above] {$\rho_\beta$};
	\node (r0) at ( 0,  0) {};
   \node (s0) at (0, 0.5) {}; 
   \node (s1) at ( 0.5, 0.5) {}; 
   \node (s2) at ( 0.5, 0) {}; 
   
   \node (s3) at ( 0.5, 2) {}; 
	   \node (s4) at ( 2, 2) {}; 
	   	   \node (s5) at ( 2, .5) {}; 
   % DRAW TREE
   \fill[fill=gray] (r0.center)--(s0.center)--(s1.center)--(s2.center);
   \path[draw] (r0.center)--(s0.center);
   \path[draw] (s0.center)--(s1.center);
   \path[draw] (s1.center)--(s2.center);
   \path[draw] (s2.center)--(r0.center);

	\fill[fill=gray] (s1.center)--(s3.center)--(s4.center)--(s5.center);
	   \path[draw] (s1.center)--(s5.center);
   \path[draw] (s1.center)--(s3.center);
		
	 \draw [red,thick,domain=-45:135] plot ({0.25+cos(\x)/sqrt(8)}, {0.25+sin(\x)/sqrt(8)});
    	\draw[color=black, fill=white] (r0) circle (.02);
   \draw[color=black, fill=white]  (s0) circle (.02);
   \draw[color=black, fill=white]  (s1) circle (.02);
      \draw[color=black, fill=white]  (s2) circle (.02);

	\end{tikzpicture}
				
				\caption{The determined cases including the circle $\Big\{(1/4+\rho_\alpha)^2+(1/4+\rho_\beta)^2 = 1/8\Big\}.$}
		\label{fig:speculative_domain}

\end{figure}
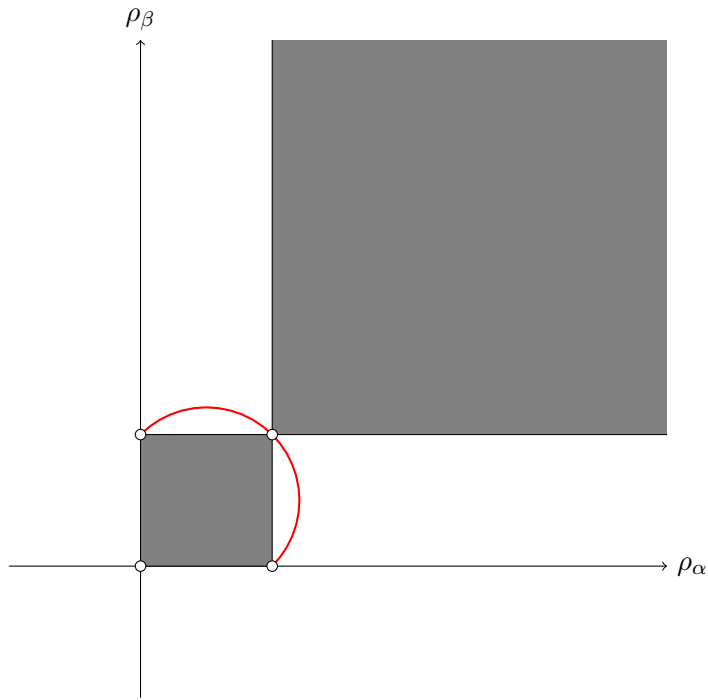

	First of all, we note the following continuity of the Widom factors. We mention in passing that a general result for the continuity of $L^p$ Widom factors with $p<\infty$ is provided in \cite[Theorem 2.1]{alpan-zinchenko20}.
	\begin{proposition}
		\label{prop:continuity_widom}
		For fixed $n\geq 0$, the Widom factor $\cW_n(\rho_\alpha,\rho_\beta)$ is continuous in $\rho_\alpha,\rho_\beta \geq 0$.
	\end{proposition}
	\begin{proof}
		Let $\tilde\theta_k^\ast$ and $\theta_k^\ast$ denote the minimizing nodes associated with the parameters $(\tilde\rho_\alpha,\tilde\rho_\beta)$ and $(\rho_\alpha,\rho_\beta)$, respectively. For any choice of parameters we have
		\begin{align*}
			\max_{x\in[-1,1]}\Big|(1-x)^{\tilde\rho_\alpha}(1+x)^{\tilde\rho_\beta}\prod_{k=1}^{n}(x-\cos\tilde\theta_k^\ast)\Big|&\leq \max_{x\in[-1,1]}\Big|(1-x)^{\tilde\rho_\alpha}(1+x)^{\tilde\rho_\beta}\prod_{k=1}^{n}(x-\cos\theta_k^\ast)\Big|
		\end{align*}
		and vice versa when the roles %r\^{o}les 
		of the parameters $(\rho_\alpha,\rho_\beta)$ and $(\tilde\rho_\alpha,\tilde\rho_\beta)$ are interchanged.
		
		Given $\varepsilon>0$ it is always possible to choose $(\tilde\rho_\alpha, \tilde\rho_\beta)$ close enough to $(\rho_\alpha,\rho_\beta)$ in a manner that, independently of the choice of $a_k\in [-1,1]$, it holds that
		\begin{align*}
		\left|\max_{x\in[-1,1]}\Big|(1-x)^{\rho_\alpha}(1+x)^{\rho_\beta}\prod_{k=1}^{n}(x-a_k)\Big|-\max_{x\in[-1,1]}\Big|(1-x)^{\tilde\rho_\alpha}(1+x)^{\tilde\rho_\beta}\prod_{k=1}^{n}(x-a_k)\Big|\right|\leq \varepsilon. \\
		\end{align*}
			We thus find that 
	\begin{align*}
			\max_{x\in[-1,1]}\Big|(1-x)^{\tilde\rho_\alpha}(1+x)^{\tilde\rho_\beta}\prod_{k=1}^{n}(x-\cos\tilde\theta_k^\ast)\Big|&\leq \max_{x\in[-1,1]}\Big|(1-x)^{\tilde\rho_\alpha}(1+x)^{\tilde\rho_\beta}\prod_{k=1}^{n}(x-\cos\theta_k^\ast)\Big| \\
			& \leq \max_{x\in[-1,1]}\Big|(1-x)^{\rho_\alpha}(1+x)^{\rho_\beta}\prod_{k=1}^{n}(x-\cos\theta_k^\ast)\Big|+\varepsilon \\
		\end{align*}
		and therefore
		\[\max_{x\in[-1,1]}\Big|(1-x)^{\tilde\rho_\alpha}(1+x)^{\tilde\rho_\beta}\prod_{k=1}^{n}(x-\cos\tilde\theta_k^\ast)\Big|-\max_{x\in[-1,1]}\Big|(1-x)^{\rho_\alpha}(1+x)^{\rho_\beta}\prod_{k=1}^{n}(x-\cos\theta_k^\ast)\Big|\leq \varepsilon.\]
		By symmetry, we may also deduce that the negative of the left-hand side is $\leq \varepsilon$
%		\[\max_{x\in[-1,1]}\left|(1-x)^{\rho_\alpha}(1+x)^{\rho_\beta}\prod_{k=1}^{n}(x-\cos\theta_k^\ast)\right|-\max_{x\in[-1,1]}\left|(1-x)^{\tilde\rho_\alpha}(1+x)^{\tilde\rho_\beta}\prod_{k=1}^{n}(x-\cos\tilde\theta_k^\ast)\right|\leq \varepsilon\]
		when $(\tilde\rho_\alpha, \tilde\rho_\beta)$ is sufficiently close to $(\rho_\alpha,\rho_\beta)$. This implies the desired continuity of $\cW_n(\rho_\alpha,\rho_\beta)$ with respect to the parameters $\rho_\alpha,\rho_\beta$. 
		
		\hfill\qedhere
	\end{proof}

	When $\rho_\alpha,\rho_\beta\in [0,1/2]$, we know from Theorem \ref{main:thm} (part 2) that \[\cW_n(\rho_\alpha,\rho_\beta)\leq 2^{1-\rho_\alpha-\rho_\beta}.\] 
	%holds, however, by adjusting the parameters $\rho_\alpha,\rho_\beta$ we know that
The same theorem (part 3) shows that	
	\[\cW_n(\rho_\alpha,\rho_\beta)\geq 2^{1-\rho_\alpha-\rho_\beta}\]
	for $\rho_\alpha,\rho_\beta \geq 1/2$. In particular, for a fixed $n$, Proposition \ref{prop:continuity_widom} implies the existence of parameters where the equality \[\cW_n(\rho_\alpha,\rho_\beta) = 2^{1-\rho_\alpha-\rho_\beta}\] is attained. In fact, such parameters exist on any arc (in the first quadrant) connecting a point of $[0,1/2]\times [0,1/2]$ with a point in $[1/2,\infty)\times[1/2,\infty)$. However, these parameters may very well be $n$ dependent. A natural question is for which parameters such an equality occurs.
	
	If we allow ourselves to assume that \eqref{eq:widom_upper_bound_lower_square} can be extended slightly outside the parameter domain $(\rho_\alpha, \rho_\beta)\in [0, 1/2] \times [0, 1/2]$, it is natural to consider parameters for which \eqref{eq:c_2} vanishes. This occurs precisely on the circular segment 
	\begin{equation}
		\left\{(\rho_\alpha,\rho_\beta) = \left(\frac{1}{4}+\frac{\cos\theta}{\sqrt{8}},\frac{1}{4} +\frac{\sin \theta}{\sqrt{8}}\right): \theta\in \left[-\frac{\pi}{4},\frac{3\pi}{4}\right] \right\}
		\label{eq:speculative_set}
	\end{equation}
	which is illustrated in Figure \ref{fig:speculative_domain}. To be specific, we recall that $M_n(\alpha,\beta)$ defined in \eqref{eq:M_n} provides an upper bound of $\cW_n(\rho_\alpha,\rho_\beta)$ and that
	\[\frac{M_{n+1}(\alpha,\beta)}{M_n(\alpha,\beta)} = \frac{(n+q+1)(n+\alpha+\beta+1)\left(n+\frac{\alpha+\beta+1}{2}\right)^{q-\frac{1}{2}}}{\left(n+\frac{\alpha+\beta+1}{2}+1\right)^{q+\frac{1}{2}}(n+\frac{\alpha+\beta}{2}+1)}.\]
	By replacing all occurrences of the natural number $n$ with the real variable $x$ on the right-hand side, we obtain the function $f(x)$ from \eqref{eq:thefunctionf}. We previously saw in \eqref{eq:logarithmic_derivative} that 
	\begin{equation*}
				\frac{f'(x)}{f(x)}=\frac{c_2(\alpha,\beta)x^2+c_1(\alpha,\beta)x+c_0(\alpha,\beta)}{(x+\alpha+1)(x+\alpha+\beta+1)(x+\frac{\alpha+\beta}{2})(x+\frac{\alpha+\beta}{2}+1)(x+\frac{\alpha+\beta}{2}+\frac{3}{2})}\end{equation*}
	and that the leading term $c_2(\alpha,\beta)x^2$ vanishes precisely when $|\alpha| = |\beta| = 1/2$. This indicates that the convergence of $f'(x)/f(x)\rightarrow 0$ as $x\rightarrow \infty$ is ``rapid'' in this case. If $\alpha^2+\beta^2>1/4$, then $f'(x)/f(x)$ will eventually be positive if $x$ is large enough. We therefore believe that the monotonicity properties of $M_{n}(\alpha,\beta)$ change somewhere in the vicinity of the circle 
	\begin{equation}
		\Big\{(\alpha,\beta):\alpha,\beta\geq 0,\, \alpha^2+\beta^2 = 1/4\Big\}.
		\label{eq:circle_first_param}
	\end{equation}
	Replacing the variables $(\alpha,\beta)$ by $(\rho_\alpha,\rho_\beta)$ transfers the circular segment \eqref{eq:circle_first_param} precisely to \eqref{eq:speculative_set}. 
	\begin{figure}[h!]
		\centering 
		\includegraphics[width = 0.8\textwidth]{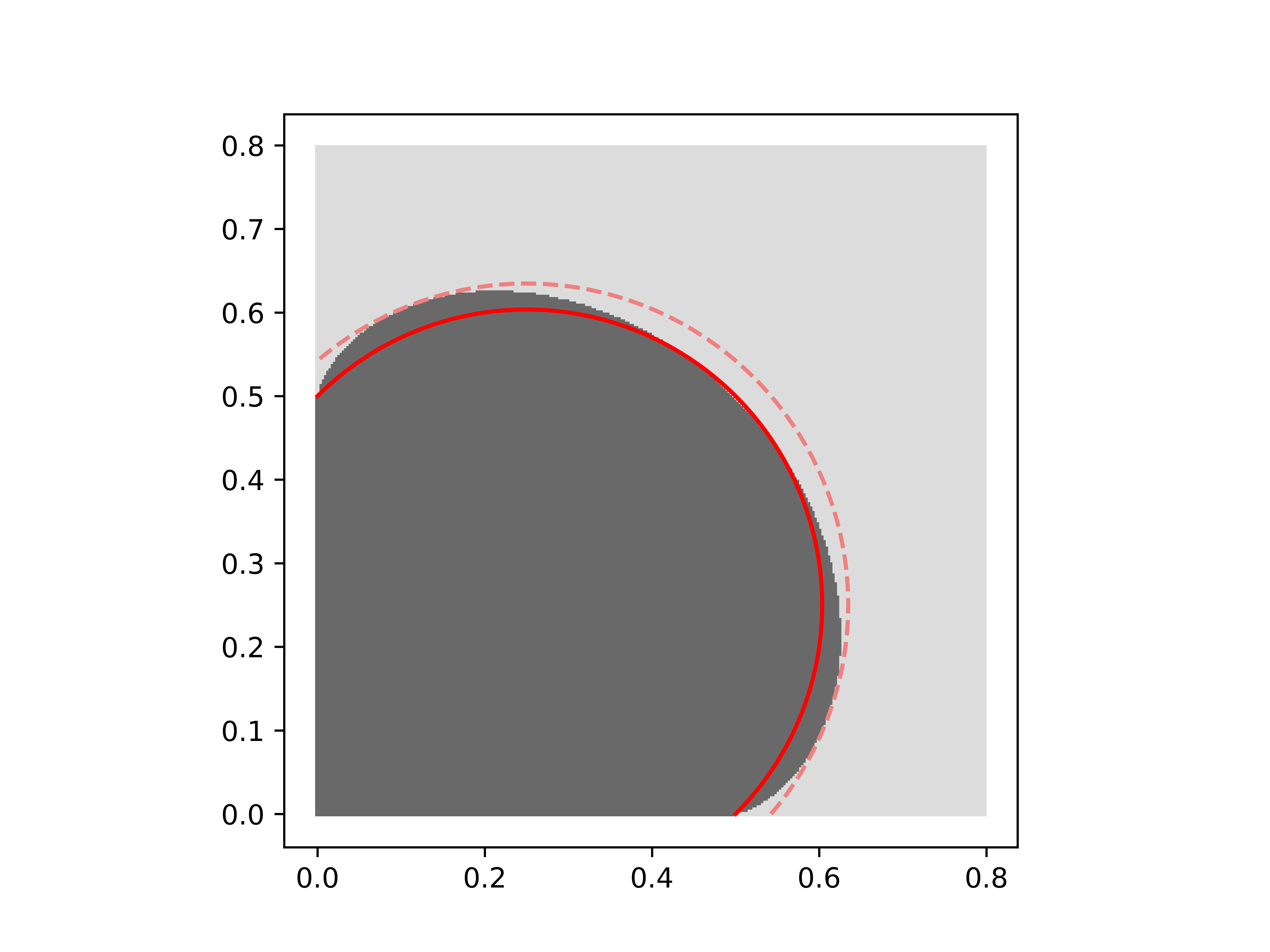}
		\caption{Numerical computations detailing the monotonicity properties of Widom factors for parameters $(\rho_\alpha,\rho_\beta)$ in the grid $[0,0.8]\times [0,0.8]$, discretized to contain 62 500 points. If monotonic increase is observed for the initial 10 Widom factors, then we color the square situated at $(\rho_\alpha,\rho_\beta)$ dark gray. If we observe monotonic decrease, the square situated at $(\rho_\alpha,\rho_\beta)$ is colored light gray. The circular segment \eqref{eq:speculative_set} is colored red, the set $\{(\rho_\alpha,\rho_\beta): (\rho_\alpha-1/4)^2+(\rho_\beta-1/4)^2 =1.1836088889/8\}$ is represented as a dotted line.} 
			\label{fig:spec}
	\end{figure}

	Since we do not know how far the inequality
	\[\cW_n(\rho_\alpha,\rho_\beta)\leq M_n(\alpha,\beta)\]
	can be extended outside of $-1/2\leq \alpha,\beta\leq 1/2$, our theoretical approach cannot be applied to study this case. We resort to a numerical investigation of the behavior of the corresponding Widom factors using a generalization of the Remez algorithm due to Tang \cite{tang87,tang88,fischer-modersitzki-93,komodromos-russell-tang95}. See also \cite{rubin24} for an overview on how this algorithm can be applied to the numerical study of Chebyshev polynomials. The numerical results are illustrated in Figure \ref{fig:spec}. It is clearly suggested by these plots that in the vicinity of the set specified in \eqref{eq:speculative_set}, the monotonicity properties of $\cW_n(\rho_\alpha,\rho_\beta)$ seem to shift from monotonically increasing to monotonically decreasing as we move from the inside of the disk to its exterior. Since these considerations are based on numerical experiments, we are left to speculate if this is indeed the true picture. We formulate the following conjecture.

	\begin{conjecture}
		Suppose that $\rho_\alpha, \rho_\beta\geq 0$.
		\begin{enumerate}
			\item If 
		\[\left(\rho_\alpha-\frac{1}{4}\right)^2+\left(\rho_\beta-\frac{1}{4}\right)^2< \frac{1}{8},\]
		then $\cW_n(\rho_\alpha,\rho_\beta)$ is monotonically increasing to $2^{1-\rho_\alpha-\rho_\beta}$ as $n\rightarrow \infty$.
		\item If
		\[\left(\rho_\alpha-\frac{1}{4}\right)^2+\left(\rho_\beta-\frac{1}{4}\right)^2> \frac{1.184}{8},\]
		then $\cW_n(\rho_\alpha,\rho_\beta)$ is monotonically decreasing to $2^{1-\rho_\alpha-\rho_\beta}$ as $n\rightarrow \infty$.
		\end{enumerate}
	\end{conjecture}

	\section{Appendix}
	For the benefit of the reader we restate the definitions of the terms $c_0(\alpha,\beta)$ and $c_1(\alpha,\beta)$ as defined in \eqref{eq:c_0} and \eqref{eq:c_1}.
	
	\begin{align*}
			c_1(\alpha,\beta)&=\frac{3\alpha^3}{4}+\frac{(4\beta+8)\alpha^2}{8}+\frac{(\beta^2-1)\alpha}{4}+\frac{\beta^3}{2}+\beta^2-\frac{\beta}{4}-\frac{1}{2},\\
			c_0(\alpha,\beta)&=\frac{\alpha^4}{4}+\frac{(3\beta+6)\alpha^3}{8}+\frac{(\beta+2)^2\alpha^2}{8}+\frac{(\beta^2-1)(\beta+2)\alpha}{8}+\frac{\beta^4}{8}+\frac{\beta^3}{2}+\frac{3\beta^2}{8}-\frac{\beta}{4}-\frac{1}{4}.
	\end{align*}
	We aim to show that these are upper bounded by $0$ in the parameter domain $\{-1/2\leq \beta\leq \alpha\leq 1/2\}$. This is a necessary part of proving Lemma \ref{lem:jacobi_estimate}.
	\begin{lemma}
		If $-1/2\leq \beta \leq \alpha \leq 1/2$, then $c_k(\alpha,\beta)\leq 0$ for $k = 1,2$ with equality if and only if $|\alpha| = |\beta| = 1/2$.
		\label{lem:coeff_analysis}
	\end{lemma}
	\begin{proof}
		To show that this is indeed true, we perform a detailed analysis of the coefficients $c_0(\alpha,\beta),\, c_1(\alpha,\beta)$ on the set $\{-1/2\leq\beta\leq\alpha\leq1/2\}$ illustrated in Figure \ref{fig:domain}.
\begin{figure}
\centering
\begin{tikzpicture}[scale=5]
	 \node (r0) at ( -0.5,  -0.5) {}; % root
   \node (s0) at (0.5, -0.5) {}; % extreme
   \node (s1) at ( 0.5, 0.5) {}; % extreme

   % DRAW TREE
   \fill[fill=gray] (r0.center)--(s0.center)--(s1.center);
   \path[draw] (r0.center)--(s0.center);
   \path[draw] (s0.center)--(s1.center);
   \path[draw] (s1.center)--(r0.center);
	\draw[color=black, fill=white] (r0) circle (.01);
   \draw[color=black, fill=white]  (s0) circle (.01);
   \draw[color=black, fill=white]  (s1) circle (.01);
   \draw (s1) node[right] {(1/2,1/2)};
   \draw (s0) node[below] {(1/2,-1/2)};
   \draw (r0) node[below] {(-1/2,-1/2)}; 

    	\draw[->] (-.75,0) -- (.75,0) node[right] {$\alpha$}; 
	    \draw[->] (0,-.75) -- (0,.75) node[above] {$\beta$};

	\end{tikzpicture}
					\caption{The parameter domain $\{-1/2\leq\beta\leq\alpha\leq1/2\}$.}
				\label{fig:domain}

\end{figure}
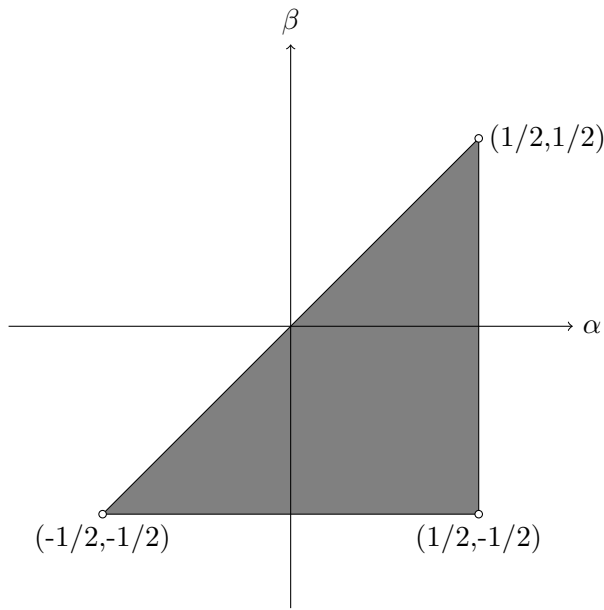
The boundary of $\{(\alpha,\beta):-1/2\leq \beta\leq \alpha\leq 1/2\}$ is parametrized with $t\in [0,1]$ through			
			\begin{align*}
				c_0(1/2,t-1/2)&=\frac{1}{8}\left(t+\frac{5}{2}\right),\left(t+1\right)t(t-1),\\
				c_0(t-1/2,-1/2)&=t(t-1)\left(\frac{1}{4}t^2+\frac{5}{16}t+\frac{1}{8}\right),\\
				c_0(t-1/2,t-1/2)&=\left(t+\frac{1}{2}\right)^2t(t-1),\\
				c_1(1/2,t-1/2)&=t(t-1)\left(\frac{t}{2}+\frac{7}{8}\right),\\
				c_1(t-1/2,-1/2)&=\frac{3}{4}t(t-1)\left(t+\frac{1}{2}\right),\\
				c_1(t-1/2,t-1/2)&=2\left(t+\frac{1}{2}\right)t(t-1).
			\end{align*}
		
			Based on these factorizations it becomes apparent that $c_0<0$ and $c_1<0$ holds for $t\in (0,1)$ but these values of $t$ are precisely those which correspond to the boundary of $\{(\alpha,\beta):-1/2\leq \beta\leq \alpha\leq 1/2\}$ with the vertices removed. 
			
			Differentiating $c_{k}(\alpha,\beta)$ for $k = 0,1$ with respect to $\alpha$ yields
			\begin{align}
				\frac{\partial}{\partial\alpha}c_0(\alpha,\beta)&= \alpha^3+\frac{9(\beta+2)}{8}\alpha^2+\frac{(\beta+2)^2}{4}\alpha-\frac{(1-\beta^2)(\beta+2)}{8},\label{eq:c_0_first}\\
				\frac{\partial^2}{\partial \alpha^2}c_0(\alpha,\beta) & = 3\alpha^2+\frac{9}{4}(\beta+2)\alpha+\frac{(\beta+2)^2}{4},\label{eq:c_0_second}\\
				\frac{\partial^3}{\partial \alpha^3}c_0(\alpha,\beta) & = 6\alpha+\frac{9}{4}(\beta+2),\label{eq:c_0_third}\\
				\frac{\partial}{\partial\alpha}c_1(\alpha,\beta)&= \frac{9}{4}\alpha^2+(\beta+2)\alpha-\frac{1-\beta^2}{4}.\label{eq:c_1_first}
			\end{align}
			
			We begin with showing that $c_0(\alpha,\beta)<0$ unless the parameters belong to the vertices of $\{(\alpha,\beta):-1/2\leq \beta\leq \alpha\leq 1/2\}$. Assuming that $\alpha,\beta\in [-1/2,1/2]$ we obtain from \eqref{eq:c_0_third} that
			\[\frac{\partial^3}{\partial \alpha^3}c_0(\alpha,\beta)=6\alpha+\frac{9}{4}(\beta+2)\geq-3-\frac{9}{8}+\frac{9}{2}=\frac{3}{8}.\]
			In particular, this implies that
			\[\alpha\mapsto\frac{\partial}{\partial\alpha}c_0(\alpha,\beta)\]
			is a convex function on $[-1/2,1/2]$ for any fixed $\beta\in [-1/2,1/2]$. Consequently, if $\beta\leq 0$ then
			\begin{equation}
				\max_{\alpha\in [\beta,0]}\frac{\partial}{\partial\alpha}c_0(\alpha,\beta)\leq \max\left\{\frac{\partial}{\partial \alpha}c_0(\beta,\beta),-\frac{(1-\beta^2)(\beta+2)}{8}\right\}.
				\label{eq:c_0_first_upper_bound}
			\end{equation}
			Clearly, the quantity
			\[-\frac{(1-\beta^2)(\beta+2)}{8}\]
			is negative if $-1/2\leq \beta\leq 1/2$. We claim that the same is true for $\frac{\partial}{\partial \alpha}c_0(\beta,\beta)$. To see this we introduce the function $h(\beta) = \frac{\partial}{\partial\alpha}c_0(\beta,\beta)$ which is given by
			\begin{align*}
			h(\beta)& = \beta^3+\frac{9}{8}(\beta+2)\beta^2+\frac{(\beta+2)^2}{4}\beta-\frac{(1-\beta^2)(\beta+2)}{8} .
			\end{align*}
			Differentiating two times we obtain
			\[h''(\beta) = 15\beta+7.\]
			It is evident that $h(\beta)$ is convex if $\beta\geq-7/15$. On the other hand, if $-1/2\leq \beta\leq -7/15$ then
			\[h'(\beta) = \frac{15}{2}\beta^2+7\beta+\frac{7}{8}\leq \frac{15}{8}-\frac{49}{15}+\frac{7}{8}<0\]
			which shows that $h$ is decreasing on $[-1/2,-7/15]$. Consequently, if $-1/2\leq \beta \leq 0$ we have
			\[\frac{\partial}{\partial\alpha}c_0(\beta,\beta) = h(\beta)\leq \max \{h(-1/2),h(0)\} =\max\left\{-\frac{1}{8},-\frac{(1-\beta^2)(\beta+2)}{8}\right\}<0.\]
			By refering to \eqref{eq:c_0_first_upper_bound} we obtain
			\[\max_{\alpha\in [\beta,0]}\frac{\partial}{\partial\alpha}c_0(\alpha,\beta)<0\]
			and therefore $c_0(\alpha,\beta)$ is strictly decreasing on $[\beta,0]$. From \eqref{eq:c_0_second} we gather that if $\beta\in [-1/2,1/2]$ then for any $\alpha\geq0$,
			\[\frac{\partial^2}{\partial \alpha^2}c_0(\alpha,\beta) = 3\alpha^2+\frac{9}{4}(\beta+2)\alpha+\frac{(\beta+2)^2}{4}\geq 0.\]
			This implies that $\alpha\mapsto c_0(\alpha,\beta)$ defines a convex function on $[0,1/2]$ and since we already concluded that $\alpha\mapsto c_0(\alpha,\beta)$ is strictly decreasing on $[\beta,0]$, we find that for any fixed $\beta\in [-1/2,1/2]$
			\[\max_{\alpha\in [\beta,1/2]}c_0(\alpha,\beta) = \max\{c_0(\beta,\beta),c_0(1/2,\beta)\}\leq 0\]
			and $c_0(\alpha,\beta) =0$ if and only if $\alpha,\beta \in \{-1/2,1/2\}$. This in particular shows that $c_0(\alpha,\beta)\leq 0$ with equality if and only if $\alpha = \beta = \pm 1/2$ or $\alpha = 1/2$ and $\beta = -1/2$.
			
			We now proceed to study the coefficient $c_1(\alpha,\beta)$ and show that $c_1(\alpha,\beta)\leq 0$ holds on the parameter domain $\{(\alpha,\beta):-1/2\leq \beta\leq \alpha\leq 1/2\}$. This turns out to be straightforward. From \eqref{eq:c_1_first} we find that
			\begin{align}&\frac{\partial}{\partial \alpha}c_1(\alpha,\beta)  = \frac{9}{4}\alpha^2+\frac{(4\beta+8)}{4}\alpha+\frac{(\beta^2-1)}{4}\\& = \frac{9}{4}\left(\alpha+\frac{2\beta+4}{9}+\sqrt{\left(\frac{2\beta+4}{9}\right)^2+\frac{1-\beta^2}{9}}\right)\left(\alpha-\frac{2\beta+4}{9}+\sqrt{\left(\frac{2\beta+4}{9}\right)^2+\frac{1-\beta^2}{9}}\right).\end{align}
			Since $-1/2\leq \beta\leq \alpha\leq 1/2$, we have
			\[\left(\alpha+\frac{2\beta+4}{9}+\sqrt{\left(\frac{2\beta+4}{9}\right)^2+\frac{1-\beta^2}{9}}\right)\geq -\frac{1}{2}+\frac{2\beta+4}{9}+\frac{2\beta+4}{9}\geq -\frac{1}{2}+2\cdot \frac{3}{9} = \frac{2}{3}-\frac{1}{2}>0\]
			and consequently the sign of $\frac{\partial}{\partial \alpha}c_1(\alpha,\beta)$ coincides with that of
			\[\left(\alpha-\frac{2\beta+4}{9}+\sqrt{\left(\frac{2\beta+4}{9}\right)^2+\frac{1-\beta^2}{9}}\right).\]
			This is sufficient to conclude that the maximum of $c_1(\alpha,\beta)$ on $[\beta,1/2]$ must be attained at an endpoint of $[\beta,1/2]$ and we gather that
			\[\max_{\alpha\in [\beta,1/2]}c_1(\alpha,\beta) = \max\{c_1(\beta,\beta),c_1(1/2,\beta)\}\leq 0\]
			with equality attained if and only if $\alpha = \beta = \pm 1/2$.\hfill\qedhere	
	\end{proof}

%
%\bibliography{reference.bib}

\section*{Acknowledgement}
This study originated from discussions held during a SQuaRE meeting at the American Institute of Mathematics (AIM) facility in Pasadena, CA. We extend our heartfelt gratitude to AIM for their generous hospitality and for creating such an inspiring scientific environment.

%USE THE BELOW OPTIONS IN CASE YOU NEED AUTHOR YEAR FORMAT.
\bibliographystyle{plain}
\bibliography{references.bib}

\begin{thebibliography}{10}

\bibitem{achieser56}
N.~I. Achieser.
\newblock {\em Theory of approximation}.
\newblock Frederick Ungar publishing co. New York, 1956.

\bibitem{alpan22}
G.~Alpan.
\newblock Extremal polynomials on a {J}ordan arc.
\newblock {\em J. Approx. Theory}, 276:105708, 2022.

\bibitem{alpan-zinchenko20}
G.~Alpan and M.~Zinchenko.
\newblock On the {W}idom factors for ${L}_p$ extremal polynomials.
\newblock {\em J. Approx. Theory}, 259:105480, 2020.

\bibitem{alpan-zinchenko24}
G.~Alpan and M.~Zinchenko.
\newblock Lower bounds for weighted {C}hebyshev and orthogonal polynomials.
\newblock {\em \normalfont{\textbf{arXiv}, math.CA}}, 2408.11496, 2024.

\bibitem{antonov-holshevnikov80}
V.~A. Antonov and K~.V. Hol\v{s}hevnikov.
\newblock Estimation of a remainder of a {L}egendre inequality (russian).
\newblock {\em Vestnik Leningrad Univ. Math.}, 13:163--166, 1981.

\bibitem{baratella86}
P.~Baratella.
\newblock Bounds for the error in {H}ilb formula for {J}acobi polynomials.
\newblock {\em Atti Acc. Scienze Torino, Cl.Sci.Fis.Mat.Natur.}, 120:207--223.

\bibitem{bergman-rubin24}
A.~Bergman and O.~Rubin.
\newblock Chebyshev polynomials corresponding to a vanishing weight.
\newblock {\em J. Approx. Theory}, 301:106048, 2024.

\bibitem{bernstein30}
S.~N. Bernstein.
\newblock Sur les polyn\^{o}mes orthogonaux relatifs \`{a} un segment fini.
\newblock {\em J. Math.}, 90:127--177, 1930.

\bibitem{bernstein31}
S.~N. Bernstein.
\newblock Sur les polyn\^{o}mes orthogonaux relatifs \`{a} un segment fini (second partie).
\newblock {\em J. Math.}, 10:219--286, 1931.

\bibitem{chebyshev54}
P.~L. Chebyshev.
\newblock Th\'{e}orie des m\'{e}canismes connus sous le nom de parall\'{e}logrammes.
\newblock {\em M\'{e}m. des sav. \'{e}tr. pr\'{e}s. \`{a} l'Acad. de. St. P\'{e}tersb.}, 7:539--568, 1854.

\bibitem{chebyshev59}
P.~L. Chebyshev.
\newblock Ser les questions de minima qui se rattachent a la repr\'{e}sentation approximative des fonctions.
\newblock {\em M\'{e}m. Acad. Sci. P\'{e}tersb.}, 7:199--291, 1859.

\bibitem{chow-gatteschi-wong94}
Y.~Chow, L.~Gatteschi, and R.~Wong.
\newblock A {B}ernstein-type inequality for the {J}acobi polynomial.
\newblock {\em Proc. Am. Math. Soc.}, 121(3):703--709, 1994.

\bibitem{christiansen-eichinger-rubin23}
J.~S. Christiansen, B.~Eichinger, and O.~Rubin.
\newblock Extremal polynomials and sets of minimal capacity.
\newblock {\em Constr. Approx. (to appear)}, 2024.

\bibitem{christiansen-simon-zinchenko-I}
J.~S. Christiansen, B.~Simon, and M.~Zinchenko.
\newblock Asymptotics of {C}hebyshev polynomials, {I}. subsets of {$\R$}.
\newblock {\em Invent. Math.}, 208:217--245, 2017.

\bibitem{christiansen-simon-zinchenko-III}
J.~S. Christiansen, B.~Simon, and M.~Zinchenko.
\newblock Asymptotics of {C}hebyshev polynomials, {III}. sets saturating {S}zeg{\H{o}}, {S}chiefermayr, and {T}otik--{W}idom bounds.
\newblock {\em Oper. Theory Adv. Appl.}, 276:231--246, 2020.

\bibitem{christiansen-simon-zinchenko-IV}
J.~S. Christiansen, B.~Simon, and M.~Zinchenko.
\newblock Asymptotics of {C}hebyshev polynomials, {IV}. comments on the complex case.
\newblock {\em JAMA}, 141:207--223, 2020.

\bibitem{christiansen-simon-zinchenko-review}
J.~S. Christiansen, B.~Simon, and M.~Zinchenko.
\newblock {\em {W}idom Factors and {S}zeg{\H{o}}--{W}idom Asymptotics, a Review}, pages 301--319.
\newblock Springer International Publishing, 2022.

\bibitem{delavallepoussin10}
C.~J. de~la Vall\'{e}e-Poussin.
\newblock Sur les polyn\^{o}mes d'approximation et la repr\'{e}sentation approch\'{e}e d'un angle.
\newblock {\em Acad. Roy. de Belg. Bulletins de la Classe des Sci.}, 12, 1910.

\bibitem{erdelyi-magnus-nevai94}
T.~Erd\'{e}lyi, A.~P. Magnus, and P.~Nevai.
\newblock Generalized {J}acobi weights, {C}hristoffel functions and {J}acobi polynomials.
\newblock 25:602--614, 1994.

\bibitem{faber19}
G.~Faber.
\newblock \"{U}ber {T}schebyscheffsche {P}olynome.
\newblock {\em J. Reine Angew. Math.}, 150:79--106, 1919.

\bibitem{fejer22}
L.~Fej\'{e}r.
\newblock \"{U}ber die {L}age der {N}ullstellen von {P}olynomen, die aus {M}inimumforderungen gewisser {A}rt entspringen.
\newblock {\em Math. Ann.}, 85:41--45, 1922.

\bibitem{fischer-modersitzki-93}
B.~Fischer and J.~Modersitzki.
\newblock An algorithm for complex linear approximation based on semi-infinite programming.
\newblock {\em Numer. Algor.}, 5:287--297, 1993.

\bibitem{gautschi09}
W.~Gautschi.
\newblock How sharp is {B}ernstein's inequality for {J}acobi polynomials?
\newblock {\em Electron. Trans. Numer. Anal.}, 36:1--8, 2009.

\bibitem{goncharov-hatinoglu15}
A.~Goncharov and B.~Hatino{\u{g}}lu.
\newblock Widom factors.
\newblock {\em Potential Anal.}, 42:671--680, 2015.

\bibitem{christiansen-simon-zinchenko-II}
P~Yuditskii J.~S.~Christiansen, B.~Simon and M.~Zinchenko.
\newblock {Asymptotics of {C}hebyshev polynomials, {II}: DCT subsets of ${\mathbb{R}}$}.
\newblock {\em Duke Math. J.}, 168(2):325 -- 349, 2019.

\bibitem{komodromos-russell-tang95}
M.~Z. Komodromos, S.~F. Russell, and P.~T.~P. Tang.
\newblock Design of {FIR} filters with complex desired frequency response using a generalized {R}emez algorithm.
\newblock {\em IEEE Trans. Circuits Syst. II}, 42:274--278, 1995.

\bibitem{koornwinder-kostenko-teschl18}
T.~Koornwinder, A.~Kostenko, and G.~Teschl.
\newblock Jacobi polynomials, {B}ernstein-type inequalities and dispersion estimates for the discrete {L}aguerre operator.
\newblock {\em Adv. Math.}, 333:796--821, 2018.

\bibitem{kroo14}
A.~Kro\'{o}.
\newblock A note on strong asymptotics of weighted {C}hebyshev polynomials.
\newblock {\em J. Approx. Theory}, 181:1--5, 2014.

\bibitem{kroo-peherstorfer08}
A.~Kro\'{o} and F.~Peherstorfer.
\newblock Asymptotic representation of weighted ${L}_\infty$- and ${L}_1$-minimal polynomials.
\newblock {\em Math. Proc. Camb. Philos. Soc.}, 144:241--254, 2008.

\bibitem{lachance-saff-varga79}
M.~Lachance, E.~B. Saff, and R.~S Varga.
\newblock Inequalities for polynomials with a prescribed zero.
\newblock {\em Math. Z.}, 168:105--116, 1979.

\bibitem{lax44}
P.~D. Lax.
\newblock {Proof of a conjecture of {P}. {E}rd{\H{o}}s on the derivative of a polynomial}.
\newblock {\em Bull. Amer. Math. Soc.}, 50:509 -- 513, 1944.

\bibitem{lorch84}
L.~Lorch.
\newblock Inequalities for ultraspherical polynomials and the gamma function.
\newblock {\em J. Approx. Theory}, 40:115--120, 1984.

\bibitem{lorentz86}
G.~G. Lorentz.
\newblock {\em Approximation of Functions}.
\newblock Chelsea publishing company, New York, 1986.

\bibitem{lubinsky-saff87}
D.~.S Lubinsky and E.~B. Saff.
\newblock Strong asymptotics for ${L}_p$ extremal polynomials $(1<p\leq \infty)$ associated with weights on $[-1, 1]$.
\newblock In E.~B. Saff, editor, {\em Approximation Theory, Tampa}, pages 83--104. Springer Berlin Heidelberg, 1987.

\bibitem{markov84}
A.~A. Markov.
\newblock Finding the smallest and the largest values of some function that deviates least form zero.
\newblock {\em Soobshch. Kharkov Soc. Math.}, 1, 1884.

\bibitem{mason93}
J.~C. Mason.
\newblock Chebyshev polynomials of the second, third and fourth kinds in approximation, indefinite integration, and integral transforms.
\newblock {\em J. Comput. Appl. Math.}, 49:169--178, 1993.

\bibitem{novello-schiefermayr-zinchenko21}
G.~Novello, K.~Schiefermayr, and M.~Zinchenko.
\newblock Weighted {C}hebyshev polynomials on compact subsets of the complex plane.
\newblock In {\em From operator theory to orthogonal polynomials, combinatorics, and number theory}, pages 357--370. Springer, 2021.

\bibitem{polya-szego}
G.~Poly\'{a} and G.~Szeg\H{o}.
\newblock {\em Problems and theorems in analysis I}.
\newblock Die {G}rundlehren der mathematischen {W}issenschaften in {E}inzeldarstellungen. Berlin, Heidelberg, Springer, 1972.

\bibitem{rubin24}
Olof Rubin.
\newblock Computing chebyshev polynomials using the complex remez algorithm.
\newblock {\em \normalfont{\textbf{arXiv}, math.CA}}, 2405.05067, 2024.

\bibitem{smirnov-lebedev68}
V.~I. Smirnov and N.~A. Lebedev.
\newblock {\em Functions of a complex variable: constructive theory}.
\newblock The M.I.T. press, Massachusetts Institute of Technology, Cambridge, Massachusetts, 1968.

\bibitem{szego75}
G.~Szeg\H{o}.
\newblock {\em Orthogonal Polynomials}.
\newblock Vol {\bf{23}.} American Math. Soc: Colloquium publ. American Mathematical Society, 1975.

\bibitem{tang87}
P.~T.~P. Tang.
\newblock {\em Chebyshev approximation on the complex plane}.
\newblock PhD thesis, University of California at Berkeley, 1987.

\bibitem{tang88}
P.~T.~P. Tang.
\newblock A fast algorithm for linear complex {C}hebyshev approximation.
\newblock {\em Math. Comp.}, 51:721--739, 1988.

\bibitem{widom69}
H.~Widom.
\newblock Extremal polynomials associated with a system of curves in the complex plane.
\newblock {\em Adv. Math.}, 3:127--232, 1969.

\end{thebibliography}

\end{document}